\documentclass[preprint,EJP]{ejpecp} 

\RequirePackage[scr]{rsfso}
\RequirePackage[mathcal]{euscript}
\RequirePackage{dsfont}

\SHORTTITLE{FCLTs for Constrained Mittag-Leffler Ensemble in Hard Edge Scaling}
\TITLE{Functional Central Limit Theorems for Constrained\\ Mittag-Leffler Ensemble in Hard Edge Scaling}

\AUTHORS{Sergey~Berezin\footnote{Department of Mathematics, Katholieke Universiteit Leuven; St. Petersburg Department of V.A. Steklov Mathematical Institute of RAS. \EMAIL{sergey.berezin@kuleuven.be; berezin@pdmi.ras.ru}}\orcid{0000-0003-4885-7794}}


\AMSSUBJ{60B20; 60B10; 60F17} 

\SUBMITTED{November 5, 2024} 
\ACCEPTED{August 27, 2025} 

\VOLUME{0}
\YEAR{2025}
\PAPERNUM{0}
\DOI{10.1214/YY-TN}

\ABSTRACT{We consider the hard-edge scaling of the Mittag-Leffler ensemble confined to a fixed disk inside the droplet. Our primary emphasis is on fluctuations of rotationally-invariant additive statistics that depend on the radius and thus give rise to radius-dependent stochastic processes. For the statistics originating from bounded measurable functions, we establish a central limit theorem in the appropriate functional space. By assuming further regularity, we are able to extend the result to a vector functional central limit theorem that additionally includes the first hitting ``time'' of the radius-dependent statistic. The proof of the first theorem involves an approximation by exponential random variables alongside a coupling technique. The proof of the second result rests heavily on Skorohod's almost sure representation theorem and builds upon a result of Galen Shorack~(1973).}

\newcommand{\suml}{\sum\limits}
\newcommand{\supl}{\sup\limits}
\newcommand{\infl}{\inf\limits}
\newcommand{\maxl}{\max\limits}
\newcommand{\minl}{\min\limits}
\newcommand{\intl}{\int\limits}
\newcommand{\liml}{\lim\limits}

\newcommand{\Leb}{\mathrm{Leb}}

\newcommand{\re}{\mathop{\mathrm{Re}}}

\newcommand{\df}{\overset{\mathrm{def}}{=}}
\newcommand{\eqas}{\overset{\mathrm{a.s.}}{=}}
\newcommand{\eqd}{\overset{\mathrm{d}}{=}}

\newcommand{\prob}[2][]{\mathcal{P}_{#1}\! \ifstrempty{#2}{}{\left\{#2\right\}}}
\newcommand{\mexp}[2][]{\mathbb E_{#1}\! \left[#2\right]\!}

\newcommand{\limninf}{\underset{n \to \infty}{\longrightarrow}}
\newcommand{\limundninf}{\lim\limits_{n \to \infty}}
\newcommand{\limninfprob}{\overset{\mathcal{P}}{\underset{n \to \infty}{\longrightarrow}}}
\newcommand{\limninfdist}{\overset{\mathrm{d}}{\underset{n \to \infty}{\longrightarrow}}}
\newcommand{\limninffidi}{\overset{\mathrm{fidi}}{\underset{n \to \infty}{\longrightarrow}}}
\newcommand{\limninfms}{\overset{\mathrm{m.s.}}{\underset{n \to \infty}{\longrightarrow}}}
\newcommand{\limninfas}{\overset{\mathrm{a.s.}}{\underset{n \to \infty}{\longrightarrow}}}

\newcommand{\id}{\mathrm{id}}

\renewcommand{\phi}{\varphi}
\renewcommand{\epsilon}{\varepsilon}
\renewcommand{\kappa}{\varkappa}

\begin{document}

\section{Introduction}

Determinantal point processes, being of central importance in contemporary random matrix theory and beyond, exhibit a rich analytical structure and have remained an active research topic in probability and related fields for almost half a century. One common way of specifying probability laws of point processes is by prescribing distributions of the corresponding linear (or additive) \textit{functionals}, also referred to as \textit{statistics}. This status, and the fact that often in practice statistics are the only quantities that can be observed directly, renders them the objects of intrinsic mathematical and applied interest.

While linear statistics of one-dimensional determinantal point processes have been extensively studied, and the corresponding theory includes a multitude of deep results, the same cannot yet be said about their multidimensional cousins (e.g., see the survey~\cite{F2023} by P. J. Forrester and references therein). This fact can be attributed to the immanent difficulty due to a much more involved analytical and algebraic structure of the objects in question. Perhaps the simplest subclass of the processes, still rich enough to demonstrate non-trivial behavior, consists of those possessing \textit{radial}, or \textit{rotational}, symmetry. To take full advantage of the symmetry, the statistics should also enjoy this property. We refer to~\cite{ABE2023, AKS2023, BC2022, BS2023, C2022, CL2023, FL2022, HM2013, S2020} and references therein for related studies and discussions.

Our primary attention will be directed towards a specific representative of the subclass characterized above, a two-dimensional radially-symmetric process known as the \textit{Mittag-Leffler ensemble}, more specifically its \textit{constrained} version. The Mittag-Leffler ensemble, named so because of the appearance of the Mittag-Leffler function in its limiting kernel (e.g., see Y.~Ameur, N.-G.~Kang, S.-M.~Seo~\cite{AKS2023}), is a natural extension of the Ginibre ensemble which allows for both light- and heavy-tailed potentials (for more information see, e.g., S.-S.~Byun and P.~J.~Forrester~\cite{BF2022} and the references therein). The \textit{constrained} ensemble features the fact that the particles are conditioned to stay in a given subdomain of the droplet. This is often illustrated by placing a hard wall along the boundary of the subdomain. Such models appear for their own sake in, e.g., S.-M.~Seo~\cite{S2022} and~\cite{S2020}, as well as to avoid the divergence problems for more general than the Mittag-Leffler potentials in P.~Elbau and G.~Felder~\cite{EF2005}. The presence of a hard wall has a drastic effect on the behavior of the particles. This can already be seen at the level of the equilibrium measures, which posses a new, singular component at the boundary of the droplet. We will see that that this singular component plays an important role in our results. For more motivation (including from the physical perspective) we refer the reader to the introduction and the references in the paper by Y. Ameur, C. Charlier, and J. Cronvall~\cite{ACC2024}. Specifying further, we will study behavior of rotationally-invariant linear statistics in the \textit{hard-edge} scaling regime, meaning at a distance of order~$O(1/n)$, as~$n \to \infty$, from the hard wall, where~$n$ is the number of particles. We note that other scalings are possible, e.g., the semi-hard scaling concerning the particles at a distance of order~$O(1/\sqrt{n})$ from the boundary. While we believe our methods can be extended for such cases, we do not address this here. 

We draw our inspiration from a recent series of papers~\cite{ACCL2022, ACCL2024, C2022, CL2023}, in which Y.~Ameur, C.~Charlier, J.~Cronvall, and J.~Lenells had been studying counting statistics of the Mittag-Leffler ensemble, along with its constrained version, from the purely asymptotic-analytic perspective. We also mention a relevant paper by C.~Charlier~\cite{C2023}, which analyzes large gap probabilities, a paper by C.~Charlier and S.-S.~Byun~\cite{BC2022}, which studies the characteristic polynomial of the eigenvalue moduli. A common feature of all of these papers is that the quantities of interest, such as the gap probability or the moment generating functions, can be written via the incomplete gamma function~$\gamma(a,z)$. The same observation was used earlier in a paper by P.~J.~Forrester~\cite{F1992}.

One of the results in~\cite{ACCL2022} is the asymptotics of the moment generating functions for the disk counting statistics, with emphasis on detailed precise asymptotics. A direct implication of this result is finite-dimensional central limit theorems, corresponding to different scaling regimes. In contrast, our emphasis will not be on precise asymptotics; instead, we aim at strengthening a particular hard-edge central limit theorem to make it cover a general large class of rotationally-invariant statistics. We will see that only elementary asymptotic analysis is needed for that. Our further objective is to study additive statistics associated with functions supported on the complement of a disk of \textit{varying radius}, dynamical additive statistics. This variability gives rise to a continuous-time stochastic process. We will prove a central limit theorem in the appropriate functional space for this stochastic process, which will also enable us to extract more properties of the underlying linear statistics. We mention related results, those establishing a central limit theorem for linear statistics of the infinite Ginibre ensemble, A.~I.~Bufetov, D.~Garc\'{i}a-Zelada, and Zh.~Lin~\cite{BGZL2022}, a functional central limit theorem for the radius-dependent counting statistics of the infinite Ginibre ensemble, M.~Frenzl and G.~Lambert~\cite{FL2022}, and a functional central limit theorem for the counting statistics of the sine process, A.~I.~Bufetov and A.~V.~Dymov~\cite{BD2019}. We also mention a relevant recent preprint by M.~Allard, P. J.~Forrester, S.~Lahiry, and B.~Shen~\cite{AFLS2025}, which was brought to our attention by one of the referees.

We note that morally the emergence of a Gaussian limit for linear statistics of determinantal point processes would not be surprising as it aligns with the well-known general theorem established by A.~B.~Soshnikov~\cite{S2000}. However, this theorem relies on the assumption that the variance grows to infinity, which may not hold, and even if it does, verifying this can be challenging. Additionally, there are known cases when the limit is not Gaussian at all (e.g., see~\cite{AC2024}). Verifying the assumption is especially hard if the kernel of the corresponding determinantal point process has a complicated structure. Besides, since Soshnikov's theorem is intended for general determinantal point processes and is established by the method of moments (cumulants), its application does not illuminate the probabilistic mechanism behind the convergence to the Gaussian law. Due to this, instead of appealing to Soshnikov's theorem, we embrace a remarkable fact, first discovered in a particular case of the Ginibre ensemble by E.~Kostlan in~\cite{K1992} and later extended to general rotationally-invariant determinantal point processes by J.~Ben Hough, M.~Krishnapur, Y.~Peres, and B.~Vir\'{a}g in~\cite{BHKPV2006}, that the radial behavior of particles in rotationally-invariant ensembles coincides with that of independent random variables, however, not identically distributed. This identification is part of the reason why the analysis in~\cite{ACCL2022, ACCL2024, C2022, C2023, BC2022, CL2023, F1992} was feasible in the first place. We will see that this perspective also helps to avoid elaborate asymptotic calculations, on one hand, and get more information about the hard-edge behavior of the particles, on the other. Effectively, the problem of studying linear statistics reduces to studying classical objects of probability-- sums of independent random variables. In the radius-varying case, one also discovers a beautiful connection to the theory of empirical processes (see, e.g., D.~Pollard~\cite{P1990book}). We also mention papers by G.~Akemann, J.~R.~Ipsen, and E.~Strahov~\cite{AIS2014}, and by G.~Akemann and E.~Strahov~\cite{AS2013}, which rely on the property described above in the context of products of random matrices.

The major complication in our study is that the independent random variables of interest are not identically distributed; their laws depend on the parameters in an inherently non-trivial way. Moreover, we cannot rely on a simple explicit form of the statistics, as our predecessors did in~\cite{ACCL2022, ACCL2024, BC2022, C2022, CL2023}. To overcome these obstacles we developed an approximation procedure of the independent random variables in question by exponential random variables with intensities spread uniformly over an interval. We couple these objects on the same probability space, which allows us to use simple estimates and bounds for the linear statistics in a manageable way. As a result, we are able to prove a multivariate central limit theorem for rotationally-invariant statistics associated with \textit{bounded measurable functions} and its stronger \textit{functional} version for radius-dependent statistics. Under additional assumptions of regularity of the statistics and their positivity, the latter result extends to yet another functional central limit theorem for the vector stochastic process; the entries of the vector are the linear statistic itself and its first hitting time. We recall that the first hitting ``time'' tells us when the radius-dependent statistic reaches a certain fixed value for the first time as the radius of the disk monotonically increases or decreases (depending on the situation). We will give explicit characterizations of the limiting stochastic processes in all the described cases.

It is also worth pointing out that the covariance functions of the processes in our central limit theorems do not involve derivatives, in stark contrast with those in the theorems from Y.~Ameur, H.~Hedenmalm, and N.~Makarov~\cite{AHM2011, AHM2015} and B.~Rider and B.~Vir\'{a}g~\cite{RV2007}. In particular, we do not find a Gaussian free field, which is, perhaps, not surprising since effectively we are dealing with the radial component of the Mittag-Leffler process only.

We begin by introducing some preliminary concepts and the notation.

\section{Preliminaries}
Fix a probability space~$(\Omega, \mathcal{F}, \mathcal{P})$ together with a family of finite radially-symmetric Borel measures~$\{\mu_n\}_{n \in \mathbb{N}}$ on~$\mathbb{C}$. Let~$\{\mathscr{Z}_n\}_{n \in \mathbb{N}}$ be a collection of determinantal point process on $(\Omega, \mathcal{F}, \mathcal{P})$ given by the kernels
\begin{equation}
  \label{eq:rad_sym_kern}
  K_n(z,w) = \suml_{j=0}^{n-1} a_{n,j}^2 (z \overline{w})^j, \quad n \in \mathbb{N},
\end{equation}
where
\begin{equation}
  \label{eq:rad_sym_kern_eq2}
  a_{n,j}^2 = \left(\intl_{\mathbb{C}} |z|^{2j}\, \mu_n(dz)\right)^{-1}.
\end{equation}
Since each~$K_n$ defines a positive orthogonal projector on~$L_2(\mathbb{C}, \mu_n)$, the existence of the determinantal point process~$\mathscr{Z}_n$ is guaranteed by Macchi--Soshnikov--Shirai--Takahashi's theorem. The process~$\mathscr{Z}_n$ is a particular case of an \textit{orthogonal polynomial ensemble}, and the distribution of (the positions of) the particles can be given explicitly by the formula
 
\begin{equation}
  \label{eq:uncond_measure_Pn}
  \mathbb{P}_n(dz_1, \ldots, dz_n) = \frac{1}{Z_n}\prod_{j < k} |z_k-z_j|^2\, \mu_n(dz_1) \otimes \cdots \otimes\mu_n(dz_n),
\end{equation}
where~$Z_n$ is the corresponding normalizing constant. The expression~\eqref{eq:uncond_measure_Pn} signifies that~$\mathscr{Z}_n$ is also an instance of a normal matrix model, related to non-Hermitian random matrices, or equally, an example of a 2D Coulomb gas. For more information, we refer the reader to the previous section (in particular, see H.~Hedenmalm and N.~Makarov~\cite{HM2013}).

If one specifies~$\mu_n$ in the following manner,
\begin{equation}
	\label{eq:MLmeasure}
	\mu_n(dz) = |z|^{2 \alpha} e^{-n |z|^{2b}} \Leb(dz), \quad \alpha> -1, b>0,
\end{equation}
where~$\Leb$ is the Lebesgue measure on~$\mathbb{C}$, then~$\mathscr{Z}_n$ becomes a \textit{Mittag-Leffler ensemble}. In contrast, if we set
\begin{equation}
  \label{eq:mittag_lef_restr_meas}
  \mu_n(dz) = |z|^{2 \alpha} e^{-n |z|^{2b}}\mathds{1}[|z| \le \rho]\, \Leb(dz),
\end{equation}
where~$\rho>0$, the particles will be forced to stay in the (closed) disk of radius~$\rho$ centered at the origin, and we end up with the \textit{constrained} Mittag-Leffler ensemble.

Recall that according to the law of large numbers for the (unconstrained) Mittag-Leffler ensemble (see~\cite{HM2013}, in particular, Remark~2.10 (iii)), the corresponding normalized empirical measure~$\nu_n$,
\begin{equation}
  \nu_n = \frac{1}{n} \suml_{z \in \mathscr{Z}_n} \delta_{z},
\end{equation}
converges weakly in probability to the equilibrium measure, the measure supported on the disk of radius~$b^{-\frac{1}{2b}}$ centered at the origin (e.g., see~Y.~Ameur, C.~Charlier, J.~Cronvall, and J.~ Lenells~\cite{ACCL2022}), and given by
\begin{equation}
	\frac{1}{\pi} b^2 |z|^{2b-2}\, \Leb(dz).
\end{equation}
Just like in~\cite{ACCL2022}, we focus on the situation when~$\rho \in (0,b^{-\frac{1}{2b}})$, which will correspond to placing a ``hard wall'' in the bulk of the spectrum for the original Mittag-Leffler ensemble. From now on, the symbol~$\mathscr{Z}_n$ will stand for the restricted Mittag-Leffler ensemble (determinantal point process), associated with~\eqref{eq:mittag_lef_restr_meas}, where~$\rho \in (0,b^{-\frac{1}{2b}})$. 

The equilibrium measure corresponding to the restricted Mittag--Leffler ensemble (e.g., see~\cite{ACCL2022}) is supported, as expected, on the disk of radius~$\rho$ centered at the origin and is given by
\begin{equation}
	\frac{1}{\pi} b^2 |z|^{2b-2} \Leb(dz) + \kappa \sigma_\rho(dz), 
\end{equation}
where~$\sigma_\rho(dz)$ is the normalized uniform measure on the circle~$\{|z| = \rho\}$ and~$\kappa = 1 - b \rho^{2b}$. We note the presence of the second term, singular with respect to the first. Correspondingly, the restricted process~$\mathscr{Z}_n$ will have a natural boundary, the circle~$\{|z| = \rho\}$, and we will be considering the hard-edge scaling regime, in which we zoom in on the interface of size~$O(1/n)$ around this boundary. Specifically, transform the process~$\mathscr{Z}_n$ by
\begin{equation}
  \label{eq:def_wn_proc_hard_edge}
  \mathscr{W}_n = - 2 \kappa n \log{\frac{\mathscr{Z}_n}{\rho}},
\end{equation}
where the principal branch of the logarithm is used. 

In~\eqref{eq:def_wn_proc_hard_edge} and in Theorem~\ref{thm:thm1} we choose the constant~$\kappa>0$ just like in the paragraph above. This choice is not arbitrary. On the one hand, it makes the formulas look more transparent; on the other, it is a manifestation of the fact that for the hard-edge scaling, only those particles contribute that ``eventually end up'' on the circle~$\{|z| = \rho \}$. The complementary statement is encoded in Proposition~\ref{lem:behav_part_le1}.

Note that the mapping in~\eqref{eq:def_wn_proc_hard_edge} sends the disk of radius~$\rho$ into the right half-plane, and the hard-edge interface becomes a strip of size~$O(1)$, as~$n \to \infty$, to the right of the imaginary axis.

We mention that one can consider several other scenarios, the semi-hard edge scaling or the bulk scaling regime, or even the case~$\rho = b^{-\frac{1}{2b}}$ (see~\cite{ACCL2022}). We believe all these situations can be analyzed with the same methods that we develop. In particular, we conjecture that for the bulk and semi-hard edge scalings, one has (mixtures) of truncated Gaussian random variables in place of the exponential random variables in Proposition~\ref{lem:behav_part_ge1}. Also, we believe that the approach developed in this paper can be generalized to a larger class of radially symmetric measures~$\mu_n$, not just those of type~\eqref{eq:MLmeasure}. We will leave this as a subject of our future research.

The central object of our study is the (normalized) \textit{time-dependent linear statistic}
\begin{equation}
  \label{eq:additive_stat}
  S_n(t) = \frac{1}{n}\sum_{w \in \mathscr{W}_n} \phi(\re{w}) \cdot \mathds{1}[\re{w} \le t],
\end{equation}
where~$\phi$ is a bounded measurable function. We think of~\eqref{eq:additive_stat} as a continuous-time stochastic process $(S_n(t), t \ge 0)$. It is easy to realize that the parameter~$t$ represents the exponentially transformed radius of the disk in terms of the original ensemble~$\mathscr{Z}_n$.

Recall that~$D[0,L]$ and~$D[0,L)$, where we allow for~$L=+\infty$, are the space of \textit{c\`{a}dl\`{a}g} functions (right continuous with left limits) on~$[0,L]$ and on~$[0,L)$, respectively. In contrast to the compact-interval case, the elements of~$D[0,L)$ do not necessarily have finite left limits at~$L$. Unless stated otherwise, we equip~$D[0,L]$ and~$D[0,L)$ with the topology of~\textit{uniform convergence} and the topology of~\textit{uniform convergence on compact sets}, respectively. For shortness, we will refer to the former as the \textit{uniform topology} and to the latter as the \textit{locally uniform topology}. 

By default, we equip both the spaces with the cylindrical $\sigma$-algebra, as it is customary. Since the spaces are non-separable,  the cylindrical (or the projection) $\sigma$-algebra is not necessarily the same as the Borel $\sigma$-algebra. This somewhat complicates the matter on the technical level (e.g., continuous functions are no longer necessarily measurable) yet still can be resolved, as pointed out by R. M. Dudley~\cite{Dudley1966,Dudley1967} (see also D. Pollard~\cite{P2011book}).

Initially, it may seem natural to consider~$S_n$ as an element of space~$D[0,+\infty]$, since $S_n(+\infty)$ is a well-defined random variable. We will discover further that a more appropriate choice is~$D[0,+\infty)$ (see Remark~\ref{rem:rem3.3}).

If we assume that~$\phi>0$, then~$S_n$ in~\eqref{eq:additive_stat} becomes a non-decreasing stochastic process. It is of interest to understand how far in time~$t$ one needs to go so that the statistic~$S_n(t)$ reaches a certain fixed level for the first time in its history. This leads to a notion of the \textit{first-hitting time}, which is defined as
\begin{equation}
  \label{eq:Qn_def}
  Q_n(h) \df \inf \{s \in [0, +\infty)|\, S_n(s) > h \},\quad h \ge 0.
\end{equation}
If for some~$h$ the set on the right-hand side of~\eqref{eq:Qn_def} is empty, then we assign~$Q_n(h) \df +\infty$. The stochastic process~$(Q_n(h), h \ge 0)$ is right continuous and thus can be considered an element of~$D[0,+\infty)$.

The goal of the subsequent sections is to prove functional limit theorems for~\eqref{eq:additive_stat} and for~\eqref{eq:Qn_def}, as~$n \to \infty$. Section~\ref{sec:FCLT} is devoted to Theorem~\ref{thm:thm1}, a functional limit theorem for~$S_n$. In Section~\ref{sec:FCLT_tightness and fidi} we prove Lemma~\ref{lem:tightness} about tightness, split~$S_n$ into two independent terms, and study each of them separately; finite-dimensional convergence of these terms, after the appropriate shift and normalization, is studied in Section~\ref{sec:beh_Sn1} and Section~\ref{sec:beh_Sn2}. The proof of Theorem~\ref{thm:thm1} is given in Section~\ref{sec:proof_thm1}. Section~\ref{sec:FCLT_hitting_time} presents the formulation and the proof of the vector functional central limit theorem for~$(S_n,Q_n)$, Theorem~\ref{thm:thm2}.

\section{Functional central limit theorem  for linear statistics}
\label{sec:FCLT}
Introduce auxiliary functions,
\begin{equation}
  \label{eq:omegas}
  \omega_1(x) = \frac{e^{x}-1-x}{x^2}e^{-x}, \quad \omega_2(x) = \frac{2(e^{x}-1-x-\frac{x^2}{2})}{x^3}e^{-x}.
\end{equation}
We note for further reference that~$\omega_1$ is a probability density on~$\mathbb{R}_+$. Our first result is the following theorem.
\begin{theorem}[Functional CLT for linear statistics]
  \label{thm:thm1}
  Let~$\{\mathscr{Z}_n\}_{n \in \mathbb{N}}$ be a collection of radially symmetric determinantal point processes defined by~\eqref{eq:rad_sym_kern}--\eqref{eq:uncond_measure_Pn}, with~$\mu_n$ specified in~\eqref{eq:mittag_lef_restr_meas} and~$\rho \in (0, b^{-\frac{1}{2b}})$. Define~$\mathscr{W}_n$ by~\eqref{eq:def_wn_proc_hard_edge} with~$\kappa= 1-b\rho^{2b}$, and consider the linear statistic~\eqref{eq:additive_stat}, where~$\phi$ is a bounded measurable function. Set
  \begin{equation}
    \label{eq:def_mk_omega}
    m_k(t) = \kappa \intl_0^{t} \big(\phi(x)\big)^k \omega_1(x)\, dx, \quad k=1,2;
  \end{equation}
  and
  \begin{equation}
    m_{12}(t_1,t_2) = \kappa \intl_0^{t_1}\intl_0^{t_2} \phi(x_1) \phi(x_2)\, \omega_2(x_1+x_2)\, dx_2\, dx_1.
  \end{equation}
  Then, there exists an a.s. bounded continuous centered Gaussian process $(G(t), t \ge 0)$ of covariance~$\mexp{G(t_1)G(t_2)} = m_2(t_1 \wedge t_2) - m_{12}(t_1,t_2)$, and the following holds
  \begin{equation}
    \label{eq:conv_skor_th_fclt1}
    \sqrt{n} \big(S_n- \mexp{S_n}\big) \limninfdist G
  \end{equation}
  in the sense of convergence in distribution of random elements of~$D[0,+\infty)$ equipped with the locally uniform topology. Moreover,
  \begin{equation}
    \label{eq:thm1_conv_of_first_moment}
    \mexp{S_n} \limninf m_1
  \end{equation}
  uniformly on compact subsets of~$[0,+\infty)$.
\end{theorem}
\begin{remark}
  \label{rem:rem1}
  The convergence of the finite-dimensional distributions is trivially implied by the theorem. In particular, setting~$\phi=1$ yields the finite-dimensional central limit theorem in Y.~Ameur, C.~Charlier, J.~Cronvall, and J.~Lenells~\cite[Corollary~1.5]{ACCL2022}.

  Furthermore, due to Slutsky's theorem, one also has the law of large numbers
  \begin{equation}
    S_n \limninfprob m_1,
  \end{equation}
  uniformly on compact subsets of~$[0,+\infty)$, in probability. In other words, for each compact~$K \subset [0,+\infty)$ and for every~$\epsilon>0$, we have
  \begin{equation}
    \prob{\| S_n - m_1\|_K > \epsilon} \limninf 0,
  \end{equation}
  where~$\|\cdot \|_K$ is the supremum norm over~$K$.
\end{remark}
\begin{remark}
  \label{rem:conv_of_moments}
  The theorem implies that all the finite-dimensional moments of~$X_{n} = \sqrt{n}\big(S_n - \mexp{S_n}\big)$ converge to those of~$G$,
  \begin{equation}
    \label{eq:conv_of_moments}
    \mexp{(X_n(t_1))^{p_1} \cdots  (X_n(t_\ell))^{p_\ell}} \limninf \mexp{(G(t_1))^{p_1}  \cdots  (G(t_\ell))^{p_\ell}},
  \end{equation}
  where~$t_1, \ldots, t_\ell \ge 0$ and~$p_1, \ldots, p_\ell \in \mathbb{N}$. Indeed, due to the independence, one can verify that all even moments~$\mexp{|X_n(t)|^{2p}}$ are bounded by a constant~$C_p>0$ that only depends on~$p \in \mathbb{N}$. Hence, the Cauchy--Bunyakovsky--Schwarz inequality guarantees that all the moments on the left-hand side of~\eqref{eq:conv_of_moments} are bounded uniformly in~$n$, and thus the corresponding sequence of the random variables is uniformly integrable. The convergence of finite-dimensional distributions concludes the argument.
\end{remark}
\begin{remark}
  \label{rem:rem3.3}
  It will be explained in Section~\ref{sec:proof_thm1} that the convergence in~$D[0,+\infty)$ in the theorem cannot be extended to that in~$D[0,+\infty]$. This happens because part of the probability mass ``escapes'' to infinity (see Proposition~\ref{lem:behav_part_le1}).
\end{remark}
\begin{remark}
	We emphasize that the result holds for arbitrary bounded measurable~$\phi$ without additional regularity, which is also reflected in the fact that the covariance function~$\mexp{G(t_1)G(t_2)}$ does not depend on the derivatives of~$\phi$. In particular, we do not identify a Gaussian free field in the limit.
\end{remark}

The proof of Theorem~\ref{thm:thm1} will rely on proving the finite-dimensional convergence and then the tightness of the corresponding family of distributions. The tightness will follow from an auxiliary inequality we derive in Lemma~\ref{lem:tightness}. The finite-dimensional convergence will follow as a consequence of the fact that the radial behavior of the particles of~$\mathscr{Z}_n$ coincides with that of independent random variables. We will show that, as $n$ becomes large, approximately~$1-\kappa = b \rho^{2b} \in (0,1)$ fraction of particles escapes to infinity; the remaining particles behave asymptotically (in the hard-edge scaling~\eqref{eq:def_wn_proc_hard_edge}) in such a way that their appropriately scaled logarithms are independent \textit{exponential} random variables with the rate parameter spread over~$[0,1]$ uniformly. By coupling on the same probability space the exponential random variables with the random variables that describe the ensemble~$\mathscr{W}_n$, we will be able to establish finite-dimensional convergence of~$\sqrt{n}\big(S_n - \mexp{S_n}\big)$ using much simpler exponential random variables, with simple dependence on the parameters.

To begin, we recall a result by J.~Ben Hough, M.~Krishnapur, Y.~Peres, and B.~Vir\'{a}g~\cite{BHKPV2006} about the behavior of radii of particles in a radially-symmetric determinantal point process. We state this theorem in a slightly modified but equivalent form.
\begin{theorem}[Ben Hough, Krishnapur, Peres, Vir\'{a}g, 2006]
  \label{thm:indep_rv_representation}
  Let~$\mathscr{Z}$ be a determinantal point process on~$(\Omega, \mathcal{F}, \mathcal{P})$ given by the kernel
  \begin{equation}
    K(z,w) = \suml_{j=0}^{n-1} \lambda_j a_{j}^2 (z \overline{w})^j,\quad \lambda_j \in [0,1],
  \end{equation}
  with respect to a finite radially-symmetric Borel measure~$\mu$ on~$\mathbb{C}$, where
  \begin{equation}
  \label{eq:310}
    a_{j}^2 = \left(\intl_{\mathbb{C}} |z|^{2j}\, \mu(dz)\right)^{-1}.
  \end{equation}
  Then,
  \begin{equation}
    \label{thm:indep_rv_repr_claim}
    \suml_{z \in \mathscr{Z}} \delta_{|z|} \eqd \suml_{j=1}^{n} I_j \, \delta_{\Upsilon_j},
  \end{equation}
  where~$\{\Upsilon_1, \ldots, \Upsilon_{n} \}$ and $\{I_1, \ldots, I_{n} \}$ are two independent sets of mutually independent random variables on~$(\Omega, \mathcal{F}, \mathcal{P})$ such that
  \begin{equation} 
	  \label{eq:312}
    \prob{\Upsilon_j \le r} = a_{j-1}^2  \intl_{\{|z| \le r\}}|z|^{2 (j-1)} \mu(dz), \quad j=1, \ldots,n,
  \end{equation}
  and the~$I_j$ are Bernoulli random variables, $\prob{I_j=1} = \lambda_{j-1}$.
\end{theorem}

\begin{remark}
	\label{rem:Gernot}
	It is worth noticing that for the symplectic ensembles with rotationally-invariant weights, a similar result holds with~$j$ replaced by~$2j$ on the right-hand side of~\eqref{eq:310} and~\eqref{eq:312}, see G.~Akemann, J. R. Ipsen, and E. Strahov~\cite[Theorem~3.6]{AIS2014}.  
\end{remark}

Rewrite~\eqref{eq:additive_stat} in terms of the original process~$\mathscr{Z}_n$,
\begin{equation}
  S_n(t) = \frac{1}{n}\sum_{z \in \mathscr{Z}_n} \phi\left(-2 \kappa n \log\left|\frac{z}{\rho}\right|\right) \mathds{1}\left[- 2 \kappa n \log{\left|\frac{z}{\rho}\right|} \le t\right].
\end{equation}
Then, Theorem~\ref{thm:indep_rv_representation} implies that
\begin{equation}
  \label{eq:S_n_via_U}
  S_n(t) \eqd \frac{1}{n}\sum_{j=1}^n \phi(U_{n,j})\, \mathds{1}\!\left[U_{n,j} \le t\right],
\end{equation}
where~$U_{n,j}$ are a.s. non-negative random variables
\begin{equation}
  \label{eq:conn_r_and_u}
  U_{n,j} = -\frac{n \kappa}{b} \log{R_{n,j}}, \quad j=1, \ldots,n;
\end{equation}
and the independent random variables~$R_{n,j}$ are distributed according to
\begin{equation}
  \label{eq:dist_r}
  \prob{R_{n,j} \in dr} = a_{n,j}^2 r^{n \rho^{2b}\theta_{n,j}-1} e^{-n \rho^{2b} r}\, \mathds{1}[0 \le r \le 1]\,  dr,\quad j=1, \ldots, n,
\end{equation}
with
\begin{equation}
  \label{eq:theta_and_anj}
  \theta_{n,j} = \frac{j + \alpha}{bn\rho^{2b}},\quad   a_{n,j}^{2} = \left(\intl_{0}^1 r^{n \rho^{2b}\theta_{n,j}-1} e^{-n \rho^{2b} r} \, dr \right)^{-1}.
\end{equation}

Now, we can establish a simple corollary of Theorem~\ref{thm:thm1}.

\begin{corollary}
  \label{cor:thm1}
  Assume that~$\phi$ has a locally bounded derivative in Theorem~\ref{thm:thm1}. Then,
  \begin{equation}
    \label{cor:cor_thm1_statement}
    \sqrt{n} (S_n- m_1) \limninfdist G
  \end{equation}
  in the sense of the convergence in distribution of random elements of~$D[0,+\infty)$ equipped with the locally uniform topology.
\end{corollary}
\begin{remark}
  Note that we have~$m_1$ in place of~$\mexp{S_n}$ in ~\eqref{cor:cor_thm1_statement}. In other words, not only are the fluctuations about local averages~$\mexp{S_n}$ Gaussian in the limit, but also, under the assumption, so are the fluctuations about the global average~$m_1$.
\end{remark}
\begin{proof}
  It suffices to show that
  \begin{equation}
    \label{eq:cor_fluct_m1_eq1}
    \sqrt{n}\big\|\mexp{S_n} - m_1\big\|_K \limninf 0,
  \end{equation}
  for every compact subset~$K \subset [0,+\infty)$, where~$\|\cdot\|_K$ stands for the supremum norm over~$K$.

  Write
  \begin{equation}
    \mexp{S_n(t)} - m_1(t) = \frac{1}{n} \suml_{j=1}^n \intl_{0}^t \phi(x) g_{n,j}(x)\, dx -  \kappa \intl_0^t \phi(x) \omega_1(x)\, dx,
  \end{equation}
  where the~$g_{n,j}$ are the densities of the~$U_{n,j}$ and~$\omega_1$ is the density given in~\eqref{eq:omegas}. Integrating by parts, we obtain
  \begin{equation}
    \begin{aligned}
      \mexp{S_n(t)} - m_1(t) = \phi(t) \big(F_n(t) - \kappa F(t) \big) - \intl_0^t \phi'(x) \big( F_n(x) - \kappa F(x) \big)\, dx,
    \end{aligned}
  \end{equation}
  where
  \begin{equation}
    F_n(t) \df \frac{1}{n} \suml_{j=1}^n F_{n,j}(t),
  \end{equation}
  with the~$F_{n,j}$ being the distribution functions of the~$U_{n,j}$,
  \begin{equation}
    F_{n,j}(t) = \intl_0^t g_{n,j}(x)\, dx,
  \end{equation}
  and
  \begin{equation}
    F(t) \df \intl_0^t \omega_1(x)\, dx.
  \end{equation}

  By using the assumed local boundedness of~$\phi'$, we arrive at
  \begin{equation}
    \label{eq:cor_fluct_m1_eq2}
    \big\|\mexp{S_n} - m_1\big\|_K \le C \big\| F_n - \kappa F \big\|_K,
  \end{equation}
  where~$C>0$ is a constant that only depends on~$K$, $\|\phi\|_K$, and~$\|\phi'\|_K$. Now, we are to prove that the right-hand side of~\eqref{eq:cor_fluct_m1_eq2} converges to zero, as~$n \to \infty$, fast enough so that~\eqref{eq:cor_fluct_m1_eq1} holds. First, note that~$F_n$ represents the expectation of a disk counting statistic, which has already been studied in Y.~Ameur, C.~Charlier, J.~Cronvall, and J.~Lenells~\cite{ACCL2022}. One can readily verify that the proof of the asymptotic formula for~$\mexp{N(r_\ell)}$ in Corollary~1.5 from~\cite{ACCL2022} holds uniformly in~$r_\ell$ on compact subsets. After passing to~$n - N(r_\ell)$ instead of~$N(r_\ell)$ and changing the variable~$r_\ell \mapsto n (1- e^{-\frac{b t}{\kappa n}})$, which does not affect the uniform convergence, we find (in our notation) that
  \begin{equation}
    F_n(t) = \kappa F(t) + O\left(\frac{\log{n}}{n}\right),
  \end{equation}
  where the~$O$-term is uniform in~$t$ on compact sets in~$[0,+\infty)$. This implies that the right-hand side of~\eqref{eq:cor_fluct_m1_eq2} converges to zero at least as fast as~$O(\log{n}/n)$, and thus~\eqref{eq:cor_fluct_m1_eq1} follows.
\end{proof}

\subsection{Tightness and finite-dimensional convergence}
\label{sec:FCLT_tightness and fidi}
We will represent~$S_n$ in~\eqref{eq:S_n_via_U} as a sum of two \textit{independent} processes,
\begin{equation}
  \label{eq:split_Sn}
  S_n(t) \eqd S_n^{(1)}(t) + S_n^{(2)}(t),
\end{equation}
where
\begin{equation}
  \label{eq:def_of_Sn12}
  S_n^{(1)}(t) = \frac{1}{n}\sum_{\theta_{n,j}\le 1} \phi(U_{n,j})\, \mathds{1}\left[U_{n,j} \le t\right], \quad S_n^{(2)}(t) = \frac{1}{n}\sum_{\theta_{n,j} > 1} \phi(U_{n,j})\, \mathds{1}\left[U_{n,j} \le t\right],
\end{equation}
where~$U_{n,j}$ is defined in~\eqref{eq:conn_r_and_u}--\eqref{eq:theta_and_anj}. We will witness subsequently that these two terms manifest very different types of behavior.

\subsubsection{Tightness}
We establish a lemma which will help us to ensure the tightness of (the distributions of) the random variables we are dealing with. To begin, recall the definition of the \textit{Skorohod topology} on~$D[a,b]$. This topology is induced by the \textit{Skorohod metric}
\begin{equation}
  \label{eq:Skorohod_metric}
  d(x,y) = \infl_{\lambda \in \Lambda} \{\|\lambda- \id \| \vee \|x - y \circ \lambda \| \}, \quad x,y \in D[a,b],
\end{equation}
where~$\| \cdot \|$ stands for the uniform norm on~$[a, b]$ and~$\Lambda$ is the class of strictly increasing continuous mappings of~$[a,b]$ onto itself. For more information about this topic see P.~Billingsley~\cite{B1999book} and D. Pollard~\cite{P2011book}.

\begin{lemma}
  \label{lem:tightness}
  Let~$(A_{n,j},B_{n,j})$, $j=1,\ldots,n$, $n \in \mathbb{N}$, be independent pairs of random variables. Assume that all~$A_{n,j}$ are bounded in absolute value by the same constant~$M>0$. Set
  \begin{equation}
    F_n(t) = \frac{1}{n} \suml_{j=1}^{n} F_{n,j}(t),
  \end{equation}
  where
  \begin{equation}
    F_{n,j}(t) = \prob{B_{n,j} \le t}, \quad j=1, \ldots,n, \ n \in \mathbb{N},
  \end{equation}
  and assume that~$F_n$ converges uniformly to a continuous function~$F$ on~$[a,b]$.

  Define the stochastic processes~$(X_n(t), t \in [a,b])$, $n \in \mathbb{N}$, by
  \begin{equation}
    X_n(t) = \sqrt{n} \big(S_n(t) - \mexp{S_n(t)}\big),  
  \end{equation}
  where
  \begin{equation}
    S_n(t) = \frac{1}{n} \suml_{j=1}^{n}A_{n,j}\, \mathds{1}[B_{n,j} \le t].
  \end{equation}

  Then, the family of distributions corresponding to~$X_n$ is tight in the space~$D[a,b]$ with the Skorohod topology, provided the following conditions are also satisfied
  \begin{gather}
    \liml_{c \to +\infty} \overline{\limundninf} \prob{|X_n(t)| > c} = 0, \quad t \in [a,b]. \label{eq:tight_prop_eq1}\\
    \liml_{t \downarrow a} \overline{\limundninf} \prob{|X_n(t) - X_n(a)| > \epsilon} = 0, \  \liml_{t \uparrow b} \overline{\limundninf} \prob{|X_n(b-) -X_n(t)| > \epsilon} = 0.    \label{eq:tight_prop_eq2}
  \end{gather}
\end{lemma}
\begin{remark}
  The lemma remains valid when $b = +\infty$, and the proof only requires minor superficial adjustments.
\end{remark}

\begin{proof}
  \noindent \textbf{a)} First we will establish an auxiliary inequality. Write
  \begin{equation}
    \label{eq:tight_proof_eq1}
    \mexp{(X_n(t)-X_n(t_1))^2(X_n(t_2)-X_n(t))^2} = \frac{1}{{n}^2}\mexp{\left(\suml_{j=1}^{n} \mathring{\alpha}_j\right)^2\left(\suml_{j=1}^{n} \mathring{\beta}_j\right)^2},
  \end{equation}
  where~$a \le t_1< t < t_2 \le b$ and
  \begin{equation}
    \alpha_{n,j} = A_{n,j} \mathds{1}[t_1 < B_{n,j} \le t],\quad     \beta_{n,j} = A_{n,j} \mathds{1}[t < B_{n,j} \le t_2]
  \end{equation}
  are two orthogonal random variables whose centered versions are defined by
  \begin{equation}
    \mathring{\alpha}_{n,j} = \alpha_{n,j} - \mexp{\alpha_{n,j}}, \quad \mathring{\beta}_{n,j} = \beta_{n,j} - \mexp{\beta_{n,j}}.
  \end{equation}
  Now, we are going to simplify the expression~\eqref{eq:tight_proof_eq1} using the independence,
  \begin{equation}
    \label{eq:tight_proof_eq2}
    \begin{aligned}
      &\mexp{(X_n(t)-X_n(t_1))^2(X_n(t_2)-X_n(t))^2} =  \frac{1}{{n}^2}\mexp{\suml_{\substack{j_1, j_2=1\\k_1,k_2=1}}^{n} \mathring{\alpha}_{n,j_1}\mathring{\alpha}_{n,j_2}\mathring{\beta}_{n,k_1}\mathring{\beta}_{n,k_2}}\\
      &=  \frac{1}{n^2} \Bigg[ \suml_{j \ne k}\mexp{\mathring{\alpha}_{n,j}^2} \mexp{\mathring{\beta}_{n,k}^2} + 2 \suml_{j \ne k}\mexp{\mathring{\alpha}_{n,j}\mathring{\beta}_{n,j}}\mexp{\mathring{\alpha}_{n,k}\mathring{\beta}_{n,k}} + \suml_{j}\mexp{\mathring{\alpha}_{n,j}^2 \mathring{\beta}_{n,j}^2} \Bigg].
    \end{aligned}
  \end{equation}
  Observe that 
  \begin{equation}
    \label{eq:tight_proof_eq3}
    \begin{aligned}
      &\mexp{\mathring{\alpha}_{n,j}^2} \mexp{\mathring{\beta}_{n,k}^2}  \le \mexp{{\alpha}_{n,j}^2} \mexp{{\beta}_{n,k}^2} \le M^4 (F_{n,j}(t)-F_{n,j}(t_1)) (F_{n,k}(t_2)-F_{n,k}(t))\\
      &\le M^4 n^2 (F_{n}(t)-F_{n}(t_1)) (F_{n}(t_2)-F_{n}(t)) \le M^4 n^2 (F_{n}(t_2)-F_{n}(t_1))^2.
    \end{aligned}
  \end{equation}
  Using the fact that~$\alpha_{n,j} \beta_{n,j} = 0$, we obtain
  \begin{equation}
    \label{eq:tight_proof_eq4}
    \begin{aligned}
      &\mexp{\mathring{\alpha}_{n,j}\mathring{\beta}_{n,j}}\mexp{\mathring{\alpha}_{n,k}\mathring{\beta}_{n,k}} = \mexp{\alpha_{n,j}}\mexp{\beta_{n,j}}\mexp{\alpha_{n,k}}\mexp{\beta_{n,k}} \\
      &\le M^4 (F_{n,j}(t_2)-F_{n,j}(t)) (F_{n,k}(t)-F_{n,k}(t_1)) \le M^4 n^2 (F_{n}(t_2)-F_{n}(t_1))^2
    \end{aligned}
  \end{equation}
  and
  \begin{equation}
    \label{eq:tight_proof_eq5}
    \begin{aligned}
      &\mexp{\mathring{\alpha}_{n,j}^2 \mathring{\beta}_{n,j}^2} = \big(\mexp{\alpha_{n,j}}\big)^2\mexp{\beta_{n,j}^2} + \big(\mexp{\beta_{n,j}}\big)^2\mexp{\alpha_{n,j}^2} - 3 \big(\mexp{\alpha_{n,j}}\big)^2\big(\mexp{\beta_{n,j}}\big)^2 \\
      &\le 2 M^4 (F_{n,j}(t_2)-F_{n,j}(t)) (F_{n,j}(t) -F_{n,j}(t_1)) \le 2 M^4 n^2 (F_{n}(t_2)-F_{n}(t_1))^2.
    \end{aligned}
  \end{equation}

  In light of~\eqref{eq:tight_proof_eq3} -- \eqref{eq:tight_proof_eq5}, we find from~\eqref{eq:tight_proof_eq2} that
  \begin{equation}
    \label{eq:tight_proof_eq6}
    \mexp{(X_n(t)-X_n(t_1))^2(X_n(t_2)-X_n(t))^2} \le 5 M^4 (F_n(t_2) - F_n(t_1))^2.
  \end{equation}
  Note that the right-hand side depends on~$n$ and we cannot use standard results such as~\cite[Theorem~13.5]{B1999book} directly to establish the tightness. In particular, this makes our setup more involved than that in, e.g., G.~R. Shorack~\cite[p.~148]{S1973}-- we deal with the original process~$X_n$ and not its ``reduced'' version.  The following argument will remedy the situation.
  \\[2ex]
  \noindent \textbf{b)} We are going to bound the modulus of ``continuity'' of~$X_n \in D[a,b]$,
  \begin{equation}
    w''(X_n,\delta) = \supl_{\substack{t_1 < t < t_2\\|t_2-t_1| < \delta}} \big(|X_n(t)-X_n(t_1)|\wedge|X_n(t_2)-X_n(t)|\big),
  \end{equation}
  where the supremum extends over all~$t_1,t,t_2 \in [a,b]$ satisfying the inequalities.
  
  An argument similar to that in the proof of Theorem~13.5 from P.~Billingsley~\cite{B1999book} shows that for every~$\epsilon>0$, the bound~\eqref{eq:tight_proof_eq6} implies
  \begin{equation}
    \prob{w''(X_n,\delta) > \epsilon} \le \frac{2C}{\epsilon^4} w(F_n, 2 \delta)
  \end{equation}
  for some constant~$C>0$, where~$w$ is the usual modulus of continuity
  \begin{equation}
    w(F_n, 2 \delta) = \supl_{|t_2-t_1|< 2 \delta}|F_n(t_2) - F_n(t_1)|;
  \end{equation}
  the supremum extends over all~$t_1,t_2 \in [a,b]$ satisfying~$|t_2-t_1| < 2\delta$.
  
  The triangle inequality yields
  \begin{equation}
    w(F_n, 2 \delta) \le 2 \| F_n-F\| + w(F, 2 \delta).
  \end{equation}
  The uniform convergence of~$F_n$ to~$F$ lets us make the first term arbitrarily small by choosing~$n$ large enough. The continuity of~$F$ on~$[a,b]$, and thus the \textit{uniform} continuity, gives control over the second term: by choosing~$\delta>0$ we can make~$w(F, 2 \delta)$ arbitrarily small. As a result, we arrive at
  \begin{equation}
    \label{eq:tightness_proof_bound_on_mod_of_cont}
    \liml_{\delta \to +0} \overline{\limundninf} \prob{w''(X_n,\delta) > \epsilon} = 0.
  \end{equation}

  The identity~\eqref{eq:tightness_proof_bound_on_mod_of_cont}, along with~\eqref{eq:tight_prop_eq1}--\eqref{eq:tight_prop_eq2}, implies the tightness (e.g., see~\cite[Theorem~13.2, Corollary, Theorem~13.3]{B1999book}) and completes the proof.
\end{proof}

\subsubsection{Behavior of~$S_n^{(1)}$}
\label{sec:beh_Sn1}
We treat each independent term in~\eqref{eq:split_Sn} separately, starting with $S_n^{(1)}$. In particular, we begin by looking into the behavior of~$U_{n,j}$ when~$\theta_{n,j} < 1-\delta$.
\begin{proposition}[Behavior of particles for~$\theta_{n,j} < 1-\delta$]
  \label{lem:behav_part_le1}
  Fix~$\delta \in (0,1)$. Then, 
  \begin{equation}
    \minl_{\theta_{n,j} < 1-\delta} U_{n,j} \limninfprob \infty,
  \end{equation}
  that is, for every~$\epsilon>0$
  \begin{equation}
    \prob{\minl_{\theta_{n,j} < 1-\delta} U_{n,j} < \frac{1}{\epsilon}} \limninf 0.
  \end{equation}
\end{proposition}
\begin{remark}
  Essentially, the proposition states that a portion of particles escapes to infinity with probability arbitrarily close to one.
\end{remark}
\begin{proof}
  Recall~\eqref{eq:conn_r_and_u}--\eqref{eq:theta_and_anj} and set
  \begin{equation}
    U_{n} = \minl_{\theta_{n,j} < 1-\delta} U_{n,j}, \quad R_{n} = \maxl_{\theta_{n,j} < 1-\delta} R_{n,j}.
  \end{equation}
  Clearly,
  \begin{equation}
    \label{eq:lem_del_les1_eq1}
    \prob{U_n < \frac{1}{\epsilon}} \le \prob{U_n < \frac{1}{\epsilon}, R_{n} \le 1-\frac{\delta}{2}} + \prob{R_{n} > 1-\frac{\delta}{2}}.
  \end{equation}
  
  For the first term on the right-hand side, the definition~\eqref{eq:conn_r_and_u} implies
  \begin{equation}
	  \prob{U_n < \frac{1}{\epsilon}, R_{n} \le 1-\frac{\delta}{2}} = \prob{-\frac{n \kappa}{b} \log{R_{n}} < \frac{1}{\epsilon}, R_{n} \le 1-\frac{\delta}{2}}.
  \end{equation}
  The~$R_{n,j}$ are bounded away from one, thus the latter probability becomes identically zero as soon as~$n$ is large enough and~$-\frac{n \kappa}{b} \log{(1-\delta/2)} \ge 1/\epsilon$ holds.

  For the second term, the union bound yields
  \begin{equation}
    \label{eq:lem_del_les1_eq11}
    \prob{R_{n} > 1-\frac{\delta}{2}} \le n \maxl_{\theta_{n,j} < 1-\delta} \prob{R_{n,j} > 1-\frac{\delta}{2}}.
  \end{equation}
  The latter probability can be calculated explicitly. From~\eqref{eq:dist_r}, we have
  \begin{equation}
    \label{eq:lem_del_les1_eq2}
    \prob{R_{n,j} > 1-\frac{\delta}{2}} = \frac{\intl_{1-\frac{\delta}{2}}^1 r^{n \rho^{2b}\theta_{n,j}-1} e^{-n \rho^{2b} r}\, dr}{\intl_0^1 r^{n \rho^{2b}\theta_{n,j}-1} e^{-n \rho^{2b}r}\, dr} = 1 - \frac{\gamma(n \rho^{2b} \theta_{n,j},n \rho^{2b}(1-\frac{\delta}{2}))}{\gamma(n \rho^{2b} \theta_{n,j} ,n \rho^{2b})},
  \end{equation}
  where~$\gamma(\cdot, \cdot)$ is the lower incomplete gamma-function,
  \begin{equation}
    \gamma(a,z) = \intl_0^z s^{a-1} e^{-s}\, ds.
  \end{equation}

  The asymptotic behavior of the last expression in~\eqref{eq:lem_del_les1_eq2} depends on the ratios
  \begin{equation}
    \lambda_{n,j} = \frac{n \rho^{2b}}{n \rho^{2b} \theta_{n,j}} = \frac{1}{\theta_{n,j}}, \quad \widetilde{\lambda}_{n,j} = \frac{n \rho^{2b}(1-\frac{\delta}{2})}{n \rho^{2b} \theta_{n,j}} = \frac{1-\frac{\delta}{2}}{\theta_{n,j}}.
  \end{equation}
  Observe that, for~${\theta_{n,j} < 1-\delta}$, the quantities~$\lambda_{n,j}$ and~$\widetilde{\lambda}_{n,j}$ stay bounded away from one, 
  \begin{equation}
    \lambda_{n,j} \wedge \widetilde{\lambda}_{n,j} > 1 + \frac{\delta}{2}.
  \end{equation}
  It follows from the known asymptotics of the incomplete gamma function, see Y.~Ameur, C.~Charlier, J.~Cronvall, and J.~Lenells~\cite[Lemma A.2(i)]{ACCL2022}, that there exist an index~$j_0$ and a small constant~$c>0$ such that for all~$j \ge j_0$ one has
  \begin{equation}
    \label{eq:lem_del_les1_eq3}
    \gamma(n \rho^{2b} \theta_{n,j} ,n \rho^{2b}) = \Gamma\Big(\frac{j + \alpha}{b}\Big) \big(1+ O(e^{-c n})\big)
  \end{equation}
  and
  \begin{equation}
    \label{eq:lem_del_les1_eq4}
    \gamma\bigg(n \rho^{2b} \theta_{n,j} ,n \rho^{2b}\Big(1-\frac{\delta}{2}\Big)\bigg) = \Gamma\Big(\frac{j + \alpha}{b}\Big) \big(1+ O(e^{-c n})\big)     
  \end{equation}
  as~$n \to \infty$, where~$\Gamma(\cdot)$ stands for the gamma function. Likewise, for~$j < j_0$, one can obtain formulas which asymptotically look just like the right-hand side of~\eqref{eq:lem_del_les1_eq3}--\eqref{eq:lem_del_les1_eq4}, see~\cite[Lemma A.1]{ACCL2022}. Plugging in the formulas into~\eqref{eq:lem_del_les1_eq2}, we end up with
  \begin{equation}
    \prob{R_{n,j} > 1-\frac{\delta}{2}} = O(e^{-c n})
  \end{equation}
  uniformly for all~$j=1,\ldots,n$ as~$n \to \infty$. Hence, by \eqref{eq:lem_del_les1_eq11} we certainly have
  \begin{equation}
    \prob{R_{n} > 1-\frac{\delta}{2}} \limninf 0,
  \end{equation}
  which completes the proof.
\end{proof}

The proposition we proved gives us a handle on the finite-dimensional convergence of~$S_n^{(1)}$, more precisely, of its shifted and normalized version.
\begin{proposition}
  \label{prop:Sn1_fidi}
  Let~$S_n^{(1)}$ be defined as in~\eqref{eq:def_of_Sn12}. Then,
  \begin{equation}
    \label{eq:prop_S1_conv_eq1}
    \sqrt{n}\Big(S_n^{(1)} - \mexp{S_n^{(1)}}\Big) \limninffidi 0.
  \end{equation}
  Moreover,
  \begin{equation}
    \label{eq:prop_S1_conv_eq2}
    \mexp{S_n^{(1)}} \limninf 0
  \end{equation}
  uniformly on compact sets in~$[0,+\infty)$.
\end{proposition}
\begin{proof}
  Denote the left-hand side of~\eqref{eq:prop_S1_conv_eq1} by~$X_n^{(1)}(t)$,
  \begin{equation}
    X_n^{(1)}(t) = \sqrt{n}\Big(S_n^{(1)}(t)-\mexp{S_n^{(1)}(t)}\Big).
  \end{equation}

  We will use Cram\'{e}r--Wold's device to prove the finite dimensional convergence of~$X_n^{(1)}$. To that end, set
  \begin{equation}
    \widetilde{S}^{(1)}_n = \suml_{k=1}^\ell c_k S_n^{(1)}(t_k), \ \widetilde{X}^{(1)}_n = \suml_{k=1}^\ell c_k X_n^{(1)}(t_k),
  \end{equation}
  for some~$c_k \in \mathbb{R}$ and~$t_k \in [0,+\infty)$.
  
  For convenience, define
  \begin{equation}
    A_{n,j} = \phi(U_{n,j}) \suml_{k=1}^\ell c_k\,  \mathds{1}[U_{n,j} \le t_k], \quad A_n = \maxl_{\theta_{n,j} \le 1-\delta} A_{n,j}.
  \end{equation}
  We will establish convergence in mean square,
  \begin{equation}
    \label{eq:prop_fidi_proof_eq2}
    \widetilde{X}^{(1)}_n\limninfms 0.
  \end{equation}

  Take arbitrary~$\delta \in (0,1)$. The independence allows us to write
  \begin{equation}
    \label{eq:prop_fidi_proof_eq111}
    \mexp{\big(\widetilde{X}^{(1)}_n\big)^2} =\frac{1}{n}\suml_{\theta_{n,j} \le 1-\delta} \!\mexp{\big(A_{n,j} - \mexp{A_{n,j}}\big)^2} + \frac{1}{n}\suml_{ 1-\delta < \theta_{n,j}  \le 1} \!\mexp{\big(A_{n,j} - \mexp{A_{n,j}}\big)^2}.
  \end{equation}
  
  Since the~$A_{n,j}$ are all bounded, observe that
  \begin{equation}
    \label{eq:prop_fidi_proof_eq112}
    \frac{1}{n}\suml_{1-\delta < \theta_{n,j}  \le 1} \mexp{\big(A_{n,j} - \mexp{A_{n,j}}\big)^2} \le C \delta
  \end{equation}
  for some constant~$C>0$ independent of~$\delta$ or~$n$.
  
  The first term in~\eqref{eq:prop_fidi_proof_eq111} admits the following estimate,
  \begin{equation}
    \label{eq:prop_fidi_proof_eq1}
    \frac{1}{n}\suml_{\theta_{n,j} \le 1-\delta} \mexp{\big(A_{n,j} - \mexp{A_{n,j}}\big)^2} \le \mexp{A_{n}^2;  U_n\ge  L} + \widetilde{C}\, \prob{U_n< L}
  \end{equation}
  for some constant~$\widetilde{C}>0$ independent of~$n$, where
  \begin{equation}
    U_n = \minl_{\theta_{n,j} \le 1-\delta} U_{n,j}.
  \end{equation}

  We choose~$L>0$ greater than~$t_1,\ldots,t_\ell$, then the expectation~$\mexp{A_{n}^2;  U_n\ge  L}$ becomes zero for all~$n$. We note in passing that this would not be the case, had we allowed for~$t_\ell=+\infty$.

  Next, we choose~$n$ large enough to make~$\prob{U_n< L}$ arbitrarily small by Proposition~\ref{lem:behav_part_le1}. This means that the left-hand side of the inequality~\eqref{eq:prop_fidi_proof_eq1} vanishes as~$n \to \infty$. Passing to the limit in~\eqref{eq:prop_fidi_proof_eq111}, we arrive at
  \begin{equation}
    \overline{\limundninf} \mexp{\big(\widetilde{X}^{(1)}_n\big)^2} \le C \delta.
  \end{equation}
  Since~$\delta$ can be arbitrarily small, we conclude that~\eqref{eq:prop_fidi_proof_eq2} holds. The desired finite-dimensional convergence is proven.
  
  It remains to show~\eqref{eq:prop_S1_conv_eq2}. Fix a small~$\delta>0$. By splitting the sum as in~\eqref{eq:prop_fidi_proof_eq111} and carrying out elementary estimates, one finds
  \begin{equation}
	  \supl_{t \in [0,L]}\left|\mexp{S_n^{(1)}(t)}\right| \le M \big(\prob{U_n \le L} + \delta\big). 
  \end{equation}
  where~$M = \supl_{x \ge 0} |\phi(x)|$. Once again, applying Proposition~\ref{lem:behav_part_le1} and recalling that~$\delta$ can be arbitrarily small conclude the argument.
\end{proof}

\subsubsection{Behavior of~$S_n^{(2)}$}
\label{sec:beh_Sn2}
Our next goal is to study the second term in~\eqref{eq:split_Sn}. This requires analyzing~$U_{n,j}$ for~$\theta_{n,j}>1$. We start by proving a simple technical lemma.
\begin{lemma}
  \label{lem:tv_bound}
  Let~$\mathbb{P}_1$ be a probability measure on a measurable space~$(\Omega, \mathcal{F})$, and let~$w$ be a $\mathbb{P}_1$-a.s. positive function integrable with respect to~$\mathbb{P}_1$. Define another probability measure by
  \begin{equation}
    \label{eq:def_of_P2}
    \mathbb{P}_2= \frac{w\, \mathbb{P}_1}{\intl_\Omega w\, d\mathbb{P}_1}.
  \end{equation}
  Then,
  \begin{equation}
    \label{eq:tv_dist_bound}
    d_{\mathrm{TV}}(\mathbb{P}_1, \mathbb{P}_2) \le 2 \intl_\Omega |1-w|\, d\mathbb{P}_1,
  \end{equation}
  where~$d_{\mathrm{TV}}(\cdot,\cdot)$ stands for the total variation distance between probability measures.
\end{lemma}
\begin{proof}
  Recall an equivalent definitions of the total variation distance,
  \begin{equation}
    d_{\mathrm{TV}}(\mathbb{P}_1, \mathbb{P}_2) = \maxl_{|\phi| \le 1} \left|\intl_{\Omega}\phi\, d\mathbb{P}_1 - \intl_{\Omega} \phi\, d\mathbb{P}_2  \right|,
  \end{equation}
  where the maximum is taken over all measurable functions bounded in absolute value by the unit. Plugging in~\eqref{eq:def_of_P2}, we find
  \begin{equation}
    d_{\mathrm{TV}}(\mathbb{P}_1, \mathbb{P}_2)  \le  \intl_\Omega \left|1-\frac{w}{\intl_\Omega w\, d\mathbb{P}_1}\right| d\mathbb{P}_1.
  \end{equation}

  The triangle inequality implies
  \begin{equation}
    \left|1-\frac{w}{\intl_\Omega w\, d\mathbb{P}_1}\right| \le |1-w| + \frac{w}{\intl_\Omega w\, d\mathbb{P}_1}\intl_\Omega |1-w|\, d\mathbb{P}_1,
  \end{equation}
  and thus~\eqref{eq:tv_dist_bound} follows.
\end{proof}

Our next goal is to introduce the key component of the further analysis, an approximation of the~$U_{n,j}$ with exponential random variables. We will also show that the~$U_{n,j}$ can be coupled on the same probability space with their approximants.

\begin{proposition}[Exponential approximation of the~$U_{n,j}$ for~$\theta_{n,j}> 1 + \delta$]
  \label{lem:behav_part_ge1}
  Let~$\psi$ be a bounded measurable function. Define the~$U_{n,j}$, $j=1,\ldots,n$, as in~\eqref{eq:conn_r_and_u} -- \eqref{eq:theta_and_anj}, and let the~$E_{n,j}$ be exponential random variables distributed according to
  \begin{equation}
    \label{eq:lem_exp_appr_eq1}
    \mathbb{P}_{n,j}(dx) = f_{n,j}(x)\, dx,\quad f_{n,j}(x) = \frac{b \rho^{2 b}}{\kappa}(\theta_{n,j}-1) e^{-\frac{b \rho^{2 b}}{\kappa}(\theta_{n,j}-1) x}, \quad x \ge 0,
  \end{equation} 
  where~$\theta_{n,j}$ is defined in~\eqref{eq:theta_and_anj}. Then, the $U_{n,j}$ and the~$E_{n,j}$, $j=1,\ldots,n$, $n \in \mathbb{N}$, can be coupled on the same probability space~$(\Omega, \mathcal{F}, \mathcal{P})$ in such a way that
  \begin{equation}
    \label{eq:lem_exp_appr_l2conv}
    \maxl_{\theta_{n,j} > 1+\delta} \mexp{(\psi(U_{n,j})- \psi(E_{n,j}))^2} \limninf 0.
  \end{equation}
\end{proposition}
\begin{remark}
  Although, the procedure of coupling involves enriching the original probability space and thus changing the original random variables, we keep the same notation for the sake of convenience, as it is customary. Moreover, without loss of generality, we assume that~$(\Omega, \mathcal{F}, \mathcal{P})$ in the theorem is the same space where all the determinantal point processes~$\left\{\mathscr{Z}_n\right\}$ live. This, in particular, will enable us to use the symbols~$\mathcal{P} [ \cdot ]$ and~$\mexp{\cdot}$ in the further exposition and avoid unnecessary complications.
\end{remark}
\begin{remark}
  The fact that~$\psi(U_{n,j})$ and~$\psi(E_{n,j})$ are close in the m.s. sense will drastically simplify the proof of Theorem~\ref{thm:thm1}, since m.s. convergence ``plays well'' with independent random variables, such as the~$\psi(U_{n,j})$ and the~$\psi(E_{n,j})$.
\end{remark}

\begin{proof}
  \noindent\textbf{a)} First, we will show that for every~$\delta>0$, one has
  \begin{equation}
    \label{eq:lem_exp_appr_eq1_first_claim}
    \maxl_{\theta_{n,j} > 1+\delta}d_{\mathrm{TV}}(U_{n,j}, E_{n,j}) \limninf 0.
  \end{equation}
  The probability distribution~$\mathbb{Q}_{n,j}$ of~$U_{n,j}$ from~\eqref{eq:dist_r} can be written as
  \begin{equation}
    \mathbb{Q}_{n,j} = \frac{w_n\, \mathbb{P}_{n,j}}{\intl_0^{+\infty} w_n\, d\mathbb{P}_{n,j}}
  \end{equation}
  with
  \begin{equation}
    w_n(x) = \exp{\left(-n \rho^{2b}\left(e^{-\frac{b x}{\kappa n}}-1+\frac{b x}{\kappa n}\right)\right)}.
  \end{equation}

  Lemma~\ref{lem:tv_bound} implies
  \begin{equation}
    \label{eq:lemma_exp_approx_eq1}
    \begin{aligned}
      &\maxl_{\theta_{n,j} > 1+\delta} d_\mathrm{TV}(U_{n,j}, E_{n,j})  \le 2 \maxl_{\theta_{n,j} > 1+\delta} \intl_{0}^{+\infty} |1-w_n(x)|\, f_{n,j}(x) \, dx.
    \end{aligned}
  \end{equation}
  Notice that
  \begin{equation}
    \label{eq:bound_omega}
    0 < w_n(x) \le 1,\quad w_n(x) \limninf 1,
  \end{equation}
  for~$x \ge 0$. Observe that the quantities~$\theta_{n,j}$ in~\eqref{eq:theta_and_anj} are always bounded from above for~$j=1,\ldots,n$; and since~$\theta_{n,j} > 1+\delta$, they are also bounded from below. This ensures
  \begin{equation}
    \label{eq:bound_on_fnj}
    |f_{n,j}(x)| \le C e^{-h x}, \quad x \ge 0,
  \end{equation}
  for some constants~$C>0$ and~$h>0$ independent of~$n$, $j$, or~$x$. The dominated convergence theorem applied to~\eqref{eq:lemma_exp_approx_eq1} establishes~\eqref{eq:lem_exp_appr_eq1_first_claim}.
  \\[2ex]
  \noindent\textbf{b)} Recall that, given~$n \in \mathbb{N}$, the random variables~$U_{n,j}, j=1,\ldots, n$, are independent. By the very nature of the problem at hand, the coupling for different~$n$ is not given or even relevant. We can choose any coupling we would like, however for simplicity and consistency we choose the \textit{independent} coupling. The same comment applies to the~$E_{n,j}$.

  Below, our goal is to couple on the same probability space the family~$\{U_{n,j}\}$ with the family~$\{E_{n,j}\}$, ensuring that the random variables~$\psi(U_{n,j})$ and~$\psi(E_{n,j})$ are close in the mean-square sense. Provided one can couple~$U_{n,j}$ with~$E_{n,j}$ in such a manner for fixed indexes~$n$ and~$j \in \{1, \ldots n\}$, the obtained couplings can be joined independently, resulting in a coupling for the full families. 

  We see that the problem reduces to constructing the coupling with given properties for fixed~$n$ and~$j$. One can think of the following direct route, first couple $U_{n,j}$ with~$E_{n,j}$ in such a way that they are close, then deduce that~$\psi({U}_{n,j})$ and~$\psi({E}_{n,j})$ are automatically close as well. Regrettably, this would require assuming additional regularity of~$\psi$, which is against our goals as we specifically focus on \textit{generic} bounded measurable functions~$\psi$. To find a way around this obstacle, we will pursue an indirect route.

Instead of the original random variables~$U_{n,j}$ and~$E_{n,j}$, we will show that~$\widetilde{U}_{n,j}$ and~$\widetilde{E}_{n,j}$, defined by
  \begin{equation}
	  \widetilde{U}_{n,j} \df \psi(U_{n,j}), \quad     \widetilde{E}_{n,j} \df \psi(E_{n,j}),
  \end{equation}
  can be coupled in such a way that the mean-square distance between them is small. Then, since the random variables $U_{n,j}$ and $\widetilde{U}_{n,j}$ are coupled in a natural way on a common probability space, relying on a well-known procedure (e.g., see~\cite[Lemma A1]{BP1979}), one can couple~$U_{n,j}, \widetilde{U}_{n,j}, \widetilde{E}_{n,j}$ on a common probability space by choosing~$U_{n,j}$ and~$\widetilde{E}_{n,j}$ to be conditionally independent given~$\widetilde{U}_{n,j}$. Now, $E_{n,j}$ and~$\widetilde{E}_{n,j}$  are also coupled in a natural way, thus we can apply the procedure from above to couple~$U_{n,j}, \widetilde{U}_{n,j}, \widetilde{E}_{n,j}, E_{n,j}$ by choosing~$U_{n,j}, \widetilde{U}_{n,j}$ and~$E_{n,j}$ to be conditionally independent given~$\widetilde{E}_{n,j}$. This defines all the four random variables~$U_{n,j}, E_{n,j}, \widetilde{U}_{n,j}$, and~$\widetilde{E}_{n,j}$ on a common probability space, and~$\widetilde{U}_{n,j} \eqas \psi(U_{n,j})$ and~$\widetilde{E}_{n,j} \eqas \psi(E_{n,j})$ are going to be close to each other by construction. We emphasize that the obtained coupling \textit{depends} on the choice of~$\psi$ (as opposed to the other approach, which we discarded to avoid additional regularity conditions). This is, however, sufficient for our goals, e.g., to prove Proposition~\ref{prop:fidi_conv} below. Once again, as it is customary, we use the same notation for the original and new (coupled) random variables. 

  Now, we just need to show that~$\widetilde{U}_{n,j}$ can be coupled with~$\widetilde{E}_{n,j}$ so that~\eqref{eq:lem_exp_appr_l2conv} is satisfied. Due to~\eqref{eq:lem_exp_appr_eq1_first_claim} there exists a strictly increasing sequence~$\{n_k \}_{k \in \mathbb{N}}$ of natural numbers such that for all~$n$ satisfying~$n > n_1$ one has
  \begin{equation}
	  \maxl_{\theta_{n,j} > 1+\delta}d_{\mathrm{TV}}(U_{n,j}, E_{n,j}) < \varepsilon_n,
  \end{equation}
  where~$\varepsilon_k = \frac{1}{m+1}$ for~$k \in \{n_m, n_m+1, \ldots, n_{m+1}-1 \}$, $m \in \mathbb{N}$. The definition of the total variation distance yields
  \begin{equation}
	  d_{\mathrm{TV}}(\widetilde{U}_{n,j}, \widetilde{E}_{n,j}) \le d_\mathrm{TV}(U_{n,j}, E_{n,j}),
  \end{equation}
and the distance~$d_{\mathrm{TV}}(\widetilde{U}_{n,j}, \widetilde{E}_{n,j} )$, in turn, bounds the Prohorov distance~$\pi(\widetilde{U}_{n,j}, \widetilde{E}_{n,j} )$, the infimum of~$\epsilon>0$ for which
  \begin{equation}
    \prob{\widetilde{U}_{n,j} \in B} \le \prob{\widetilde{E}_{n,j} \in B^{\epsilon}} + \epsilon, \quad \prob{\widetilde{E}_{n,j} \in B} \le \prob{\widetilde{U}_{n,j} \in B^{\epsilon}} + \epsilon
  \end{equation}
  hold for all Borel sets~$B$, where~$B^{\epsilon}$ denotes the~$\epsilon$-neighborhood of~$B$. For the proof of this fact see P. J.~Huber~\cite[p. 34]{H1981book}.

  Combining the above facts, we arrive at
  \begin{equation}
	  \maxl_{\theta_{n,j} > 1+\delta} \pi(\widetilde{U}_{n,j}, \widetilde{E}_{n,j} ) \le \maxl_{\theta_{n,j} > 1+\delta} d_\mathrm{TV}(U_{n,j}, E_{n,j}) < \varepsilon_n
  \end{equation}
  for~$n > n_1$.
  Due to a well-known result, e.g., see P.~Billingsley~\cite[Theorem 6.9]{B1999book}, we can couple~$\widetilde{U}_{n,j}$ and~$\widetilde{E}_{n,j}$ on the same probability space in such a way that
  \begin{equation}
    \label{eq:coupling_bound1}
    \maxl_{\theta_{n,j} > 1+\delta}  \prob{|\widetilde{U}_{n,j} - \widetilde{E}_{n,j}|>\varepsilon_n} < \varepsilon_n.
  \end{equation}
  For~$n \le n_1$ we use the independent coupling. Also, we mention an intentional (slight) abuse of notation. Whenever we discuss the total variation or Prohorov distance ``between random variables,'' we are actually referring to their distributions. We believe this does not cause any confusion and only helps with comprehension.

  To establish~\eqref{eq:lem_exp_appr_l2conv}, we set
  \begin{equation}
	  A_{n,j} \df |\widetilde{U}_{n,j} - \widetilde{E}_{n,j}|,
  \end{equation}
and observe that
  \begin{equation}
    \label{eq:exp_appr_final_eq}
    \begin{aligned}
      \maxl_{\theta_{n,j} > 1+\delta} \mexp{A_{n,j}^2} &\le \maxl_{\theta_{n,j} > 1+\delta}\left(\mexp{A_{n,j}^2; A_{n,j}>\epsilon_n} + \mexp{A_{n,j}^2; A_{n,j} \le \epsilon_n}\right)\\
      &\le C \maxl_{\theta_{n,j} > 1+\delta} \prob{A_{n,j}>\epsilon_n} + \epsilon_n^2
    \end{aligned}
  \end{equation}
  for some constant~$C>0$ independent of~$n$ or~$j$. Due to~\eqref{eq:coupling_bound1}, the formula~\eqref{eq:exp_appr_final_eq} becomes
  \begin{equation}
    \label{eq:exp_appr_final_eq1}
    \maxl_{\theta_{n,j} > 1+\delta} \mexp{A_{n,j}^2} < C \epsilon_n + \epsilon_n^2
  \end{equation}
  for all~$n$ satisfying~$n>n_1$. Passing to the limit as~$n \to \infty$ and noticing that~$\varepsilon_n \to 0$, we arrive at~\eqref{eq:lem_exp_appr_l2conv}. This completes the proof.
\end{proof}

We are ready to prove finite-dimensional convergence for~$S_n^{(2)}$.
\begin{proposition}
  \label{prop:fidi_conv}
  Assume that the conditions of Theorem~\ref{thm:thm1} are satisfied. Then, the following finite-dimensional convergence takes place,
  \begin{equation}
    \label{eq:fidi_prop_eq0}
    \sqrt{n} \Big(S_n^{(2)}-\mexp{S_n^{(2)}}\Big) \limninffidi G.
  \end{equation}
  Moreover,
  \begin{equation}
    \label{eq:fidi_prop_eq1}
    \mexp{S_n^{(2)}} \limninf m_1
  \end{equation}
  uniformly on~$[0,+\infty]$.
\end{proposition}

\begin{remark}
  Note that unlike in Proposition~\ref{prop:Sn1_fidi}, the convergence~\eqref{eq:fidi_prop_eq1} is uniform on the whole interval~$[0,+\infty]$. Besides, we can take one of the times to be~$+\infty$ in~\eqref{eq:fidi_prop_eq0}.
\end{remark}
\begin{proof}
  Write
  \begin{equation}
    S_n^{(2)} = S_n^{(2,1)} + S_n^{(2,2)},
  \end{equation}
  where
  \begin{equation}
    \begin{aligned}
      &S_n^{(2,1)} (t)= \frac{1}{n}\suml_{\theta_{n,j} > 1} \Big(\phi(U_{n,j})\, \mathds{1}[U_{n,j} \le t] - \phi(E_{n,j})\, \mathds{1}[E_{n,j} \le t]\Big), \\
      &S_n^{(2,2)} (t)= \frac{1}{n}\suml_{\theta_{n,j} > 1} \phi(E_{n,j})\, \mathds{1}[E_{n,j} \le t].
    \end{aligned}
  \end{equation}
  Set
  \begin{equation}
    X_n^{(2,j)}(t) = \sqrt{n}\Big(S_n^{(2,j)}(t)-\mexp{S_n^{(2,j)}(t)}\Big), \quad j=1,2.
  \end{equation}

  As earlier in the proof of Proposition~\ref{prop:Sn1_fidi}, we will rely on Cram\'{e}r--Wold's device. Define
  \begin{equation}
    \widetilde{S}^{(2,j)}_n = \suml_{k=1}^\ell c_k S_n^{(2,j)}(t_k), \quad \widetilde{X}^{(2,j)}_n = \suml_{k=1}^\ell c_k X_n^{(2,j)}(t_k),\quad j=1,2.
  \end{equation}
  for some~$c_k \in \mathbb{R}$.
  \\[1ex]
  \noindent \textbf{a)} First, we will establish that
  \begin{equation}
    \label{eq:thm_fidi_proof_eq2}
    \widetilde{X}^{(2,1)}_n\limninfms 0.
  \end{equation}
  Set
  \begin{equation}
    A_{n,j} = \suml_{k=1}^\ell c_k \Big(\phi(U_{n,j}) \mathds{1}[U_{n,j} \le t_k]-\phi(E_{n,j}) \mathds{1}[E_{n,j} \le t_k]\Big).
  \end{equation}
  Take an arbitrary~$\delta \in (0,1)$. The independence yields
  \begin{equation}
    \begin{aligned}
      \mexp{\left(\widetilde{X}^{(2,1)}_n\right)^2} =&\frac{1}{n}\suml_{\theta_{n,j} > 1+\delta} \mexp{\big(A_{n,j} - \mexp{A_{n,j}}\big)^2} \\
      &+ \frac{1}{n}\suml_{ 1+\delta \ge \theta_{n,j}  > 1} \mexp{\big(A_{n,j} - \mexp{A_{n,j}}\big)^2}.    
    \end{aligned}
  \end{equation}

  Just like in the proof of Proposition~\ref{prop:Sn1_fidi}, we find that the second term is bounded by~$C \delta$, with some constant~$C>0$ independent of~$n$ or~$j$. We need only prove that the first term can be made arbitrarily small by choosing~$n$. Note that
  \begin{equation}
    \frac{1}{n}\suml_{\theta_{n,j} > 1+\delta} \mexp{\big(A_{n,j} - \mexp{A_{n,j}}\big)^2} \le \maxl_{\theta_{n,j} > 1+\delta} \mexp{A_{n,j}^2},
  \end{equation}
  and Proposition~\ref{lem:behav_part_ge1} shows that the right-hand side vanishes in the limit~$n \to \infty$. We have~\eqref{eq:thm_fidi_proof_eq2}. A similar (simpler) argument gives
  \begin{equation}
    \label{eq:thm_fidi_proof_eq13}
    \mexp{{S}_n^{(2,1)}(t)} \limninf 0
  \end{equation}
  for any fixed~$t \ge 0$ (including~$t = +\infty$).
  \\[1ex]
  \noindent \textbf{b)} The next step is to prove that
  \begin{equation}
    \widetilde{X}^{(2,2)}_n \limninffidi G,
  \end{equation}
  where~$G$ is defined in the premise of Theorem~\ref{thm:thm1}.

  Observe that~$\widetilde{X}^{(2,2)}$ only involves the exponential random variables~$E_{n,j}$, whose dependence on~$n$ and~$j$ is of a simple form. To proceed further, we will rely on some explicit calculations. Recall from~\eqref{eq:theta_and_anj} that~$\theta_{n,j}= \frac{j + \alpha}{bn\rho^{2b}}$, and write
  \begin{equation}
	  \label{eq:3_109}
	  \mexp{S_n^{(2,2)}(t)} =  \frac{1}{n}\suml_{\theta_{n,j} > 1} \intl_0^t \phi(x) f_{n,j}(x) \, dx,
  \end{equation}
  where~$f_{n,j}(x)$ is defined in~\eqref{eq:lem_exp_appr_eq1}.
  Observe that
  \begin{equation}
	  \frac{1}{n} \suml_{\theta_{n,j} > 1} f_{n,j}(x)=\frac{1}{n}\suml_{\theta_{n,j} > 1} \frac{b \rho^{2 b}}{\kappa}(\theta_{n,j}-1) e^{-\frac{b \rho^{2 b}}{\kappa}(\theta_{n,j}-1) x},
  \end{equation}
  up to a finite number of terms which only contribute the amount of~$O(1/n)$, represents the Riemann sum for the integral
  \begin{equation}
	  b \rho^{2b} \intl_1^{1/(b \rho^{2b})} \frac{b \rho^{2 b}}{\kappa}(s-1) e^{-\frac{b \rho^{2 b}}{\kappa}(s-1) x}\, ds= \kappa \intl_0^1 s e^{-s x} \, ds = \kappa\, \omega_1(x).
  \end{equation}
  Thus, by changing the order of the integral and the limit for a finite~$t \ge 0$, we get
  \begin{equation}
    \label{eq:thm_fidi_proof_eq14}
    \limundninf \mexp{S_n^{(2,2)}(t)}  = m_1(t).
  \end{equation}

  To handle the case~$t = +\infty$, set
  \begin{equation}
	  H_n(x) = \frac{1}{n} \suml_{\theta_{n,j} > 1} f_{n,j}(x) - \kappa\, \omega_1(x),
    \end{equation}
    fix a small~$\delta>0$ and introduce
  \begin{equation}
	  H_n^{(1)}(x) =\frac{1}{n} \suml_{1+\delta \ge \theta_{n,j} > 1} f_{n,j}(x) - \kappa \intl_0^{\frac{b \rho^{2b}}{\kappa}\delta} s e^{-sx}\, ds
  \end{equation}
  and
  \begin{equation}
	  H_n^{(2)}(x) = \frac{1}{n} \suml_{\theta_{n,j} > 1+\delta} f_{n,j}(x) - \kappa \intl^1_{\frac{b \rho^{2b}}{\kappa}\delta} s e^{-sx}\, ds,
  \end{equation}
  and notice that
  \begin{equation}
     H_n = H_n^{(1)} + H_n^{(2)}.
  \end{equation}
  Write
  \begin{equation}
    \label{eq:thm_fidi_proof_eq22}
      \mexp{S_n^{(2,2)}(+\infty)} - m_1(+\infty) = \intl_0^{+\infty} \phi(x) H_n^{(1)}(x) \, dx + \intl_0^{+\infty} \phi(x) H_n^{(2)}(x) \, dx,
  \end{equation}
  and observe that, since~$\phi$ is bounded and the infinite integral is convergent, one has
  \begin{equation}
    \left| \intl_0^{+\infty} \phi(x) H_n^{(1)}(x) \, dx \right| \le C \delta,
  \end{equation}
  for some constant~$C>0$ independent of~$\delta$ and~$n$. Further, having in mind the bound~\eqref{eq:bound_on_fnj}, the dominated convergence implies that the second term on the right-hand side of~\eqref{eq:thm_fidi_proof_eq22} vanishes as~$n \to \infty$.

  We arrive at
  \begin{equation}
    \overline{\limundninf} \Big|\mexp{S_n^{(2,2)}(+\infty)} - m_1(+\infty)\Big| \le C \delta.
  \end{equation}
  Since~$\delta>0$ can be arbitrarily small, we find~\eqref{eq:thm_fidi_proof_eq14} with~$t = +\infty$. Similar calculations can be carried out for~$\widetilde{S}_n^{(2,2)}$, and it remains to analyze the covariance. 

  To that end, we use the independence to find that
  \begin{equation}
	  \begin{aligned}
	  	\mexp{X^{(2,2)}_n(t_1) X^{(2,2)}_n(t_2)} = &\frac{1}{n} \suml_{\theta_{n,j} > 1} \int \limits_{0}^{t_1 \wedge t_2} \phi^2(x) f_{n,j}(x) \, dx\\
		&- \frac{1}{n} \suml_{\theta_{n,j} > 1} \int \limits_{0}^{t_1} \int \limits_{0}^{t_2} \phi(x_1) \phi(x_2) f_{n,j}(x_1) f_{n,j}(x_2)\, dx_2 dx_1.
	  \end{aligned}
  \end{equation}

  The first term is of similar form to~\eqref{eq:3_109} and can be treated accordingly. The second term involves
  \begin{equation}
	  \frac{1}{n} \suml_{\theta_{n,j} > 1} f_{n,j}(x_1) f_{n,j}(x_2) = \frac{1}{n} \suml_{\theta_{n,j} > 1} \frac{b^2 \rho^{4 b}}{\kappa^2}(\theta_{n,j}-1)^2 e^{-\frac{b \rho^{2 b}}{\kappa}(\theta_{n,j}-1) (x_1+x_2)},
  \end{equation}
  which is the Riemann sum for
  \begin{equation}
	  \begin{aligned}
	  	b \rho^{2b} \int \limits_{1}^{1/(b \rho^{2b})} \frac{b^2 \rho^{4 b}}{\kappa^2}(s-1)^2 e^{-\frac{b \rho^{2 b}}{\kappa}(s-1) (x_1+x_2)}\, ds &= \kappa \intl_0^1 s^2 e^{-s (x_1+x_2)} \, ds\\
		&= \kappa\, \omega_2(x_1+x_2),
	  \end{aligned}
  \end{equation}
  as before up to a finite number of terms which only contribute the amount of~$O(1/n)$. The case when~$t_1 =+\infty$ or~$t_2=+\infty$ can be addressed as above and does not involve any new ideas. Consequently, one finds 
  \begin{equation}
    \label{eq:thm_fidi_proof_eq15}
    \begin{aligned}
      \limundninf \mexp{X^{(2,2)}_n(t_1) X^{(2,2)}_n(t_2)} =  m_2(t_1 \wedge t_2)  - m_{12}(t_1,t_2), \quad t_1,t_2 \in [0,+\infty],
    \end{aligned}
  \end{equation}
  and similar calculations carry over to~$\widetilde{X}_n^{(2,2)}$.

  If~$\limundninf \mexp{\left(\widetilde{X}_n^{(2,2)}\right)^2} = 0$, we trivially obtain~$\widetilde{X}_n^{(2,2)} \limninfprob 0$. Otherwise, Lyapunov's theorem for uniformly bounded random variables (e.g., see A.~N.~Shiryaev~\cite[Section III.\S4.I.2.]{S1966book}) implies finite-dimensional convergence to a Gaussian law with the covariance function~\eqref{eq:thm_fidi_proof_eq15}. Further, it follows from~\eqref{eq:thm_fidi_proof_eq2} and~\eqref{eq:thm_fidi_proof_eq15} that the limit of~$\sqrt{n}\Big(S_n^{(2)} - \mexp{S_n^{(2)}}\Big)$ is indeed a Gaussian process with the announced covariance function.

  Collecting the formulas~\eqref{eq:thm_fidi_proof_eq13} and~\eqref{eq:thm_fidi_proof_eq14}, we also conclude that
  \begin{equation}
    \label{eq:thm_fidi_proof_eq16}
    \mexp{S_n^{(2)}(t)} \limninf m_1(t),\quad  t \in [0, +\infty].
  \end{equation}
  \noindent \textbf{c)} The last step is to prove that~\eqref{eq:thm_fidi_proof_eq16} holds uniformly for all~$t \in [0,+\infty]$ at once. One can always represent~$\phi(x)$ as a difference of its positive and negative part, $\phi^{+}(x) = (\phi(x)) \vee 0$ and~$\phi^{-}(x) = (- \phi(x)) \vee 0$. Thus, we can think that~$\phi \ge 0$, without loss of generality. In this case, $\mexp{S_n^{(2)}}$ becomes a non-decreasing function. Set
  \begin{equation}
    \hat{F}_n(t) = \frac{\suml_{\theta_{n,j} > 1} \mexp{\phi(U_{n,j}) \mathds{1}[U_{n,j}\le t]}}{\suml_{\theta_{n,j} > 1} \mexp{\phi(U_{n,j})}}.
  \end{equation}
  The latter is a probability distribution function, which due to~\eqref{eq:thm_fidi_proof_eq16} converges pointwise,
  \begin{equation}
    \label{eq:thm_fidi_proof_eq17}
    \hat{F}_n(t) \limninf \hat{F}(t) = \frac{m_1(t)}{m_1(+\infty)},
  \end{equation}
  to another \textit{continuous} probability distribution function~$\hat{F}(t)$. This means that this convergence is uniform (e.g., see I.~A.~Ibragimov and Yu.~V.~Linnik~\cite[Lemma 5.1.1]{LI1971book}). Finally, the uniform convergence~\eqref{eq:thm_fidi_proof_eq17} implies that~$\mexp{S_n^{(2)}} \limninf m_1$ holds uniformly on~$[0,+\infty]$ as desired.
\end{proof}

\subsection{Proof of Theorem~\ref*{thm:thm1}}
\label{sec:proof_thm1}
\begin{proof}
  We already established finite-dimensional convergence in Proposition~\ref{prop:Sn1_fidi} and Proposition~\ref{prop:fidi_conv} for each term of~\eqref{eq:split_Sn}. Now, we are going to combine these results with the tightness from Lemma~\ref{lem:tightness}.
  \\[1ex]
  \noindent \textbf{a)} First, we will deal with~$S_n^{(1)}$. Define
  \begin{equation}
    F_n^{(1)}(t) = \frac{1}{n} \suml_{\theta_{n,j} \le 1} \prob{U_{n,j} \le t}.
  \end{equation}
  Set~$\phi = 1$ in Proposition~\ref{prop:Sn1_fidi}. The formula~\eqref{eq:prop_S1_conv_eq2} shows that~$F_n^{(1)} \limninf 0$ uniformly on compact sets.

  The finite-dimensional convergence~\eqref{eq:prop_S1_conv_eq1} guarantees that~\eqref{eq:tight_prop_eq1} and~\eqref{eq:tight_prop_eq2} of Lemma~\ref{lem:tightness} hold. This proves the tightness and thus establishes the weak convergence in~$D[0,L]$ equipped with the Skorohod topology. Because the limit is continuous (identically zero), this is equivalent to the weak convergence in~$D[0,L]$ with the \textit{uniform} topology. Since~$L>0$ is arbitrary, we find that
  \begin{equation}
    \label{eq:proof_thm1_sn1_dconv}
    S_n^{(1)} \limninfdist 0
  \end{equation}
  in~$D[0,+\infty)$ equipped with the locally uniform topology. We note that, because~$\limundninf F_n^{(1)}(t)$ has a jump at the end of the interval~$t = +\infty$, there is no hope for strengthening this convergence to make it hold in~$D[0,+\infty]$ with the \textit{uniform} topology (or the Skorohod topology, for that matter).
  \\[1ex]
  \noindent \textbf{b)} Now, we will deal with~$S_n^{(2)}$ and prove that
  \begin{equation}
    \label{eq:thm1_proof_s2_Dconv}
    \sqrt{n} \Big(S_n^{(2)} - \mexp{S_n^{(2)}} \Big) \limninf G
  \end{equation}
  in~$D[0,+\infty]$ with the uniform topology. Define 
  \begin{equation}
    F_n^{(2)}(t) = \frac{1}{n} \suml_{\theta_{n,j} > 1} \prob{U_{n,j} \le t}.
  \end{equation}
  By taking~$\phi = 1$ in Proposition~\ref{prop:fidi_conv} and using~\eqref{eq:fidi_prop_eq1}, we see that
  \begin{equation}
    F_n^{(2)} \limninf F
  \end{equation}
  uniformly on~$[0,+\infty]$, where
  \begin{equation}
	  F(t) = \kappa \intl_0^t \omega_1(x)\, dx,
  \end{equation}
  with~$\omega_1$ defined in~\eqref{eq:def_mk_omega}.  We again can use Lemma~\ref{lem:tightness} since the finite-dimensional convergence guarantees that~\eqref{eq:tight_prop_eq1} and~\eqref{eq:tight_prop_eq2} are satisfied. Indeed, the first condition follows from the tightness of the one-dimensional laws of~$X_{n}^{(2)}(t)$, which, in turn, follows from the one-dimensional convergence. The second conditions can be rephrased as the right-continuity of~$G$ at zero and its left-continuity at~$+\infty$ (both in probability). Since the covariance function of~$G$ is continuous on~$[0,+\infty] \times [0,+\infty]$, the conditions are fulfilled. As a result, we have the tightness in the Skorohod topology.

  Note that one has good control over the jumps of the process~$X_n^{(2)}$,
  \begin{equation}
    j(X_n^{(2)}) = \supl_{t \in [0,+\infty]}|\sqrt{n}\big(S_n^{(2)}(t) - S_n^{(2)}(t-0)\big)| \overset{\mathrm{a.s.}}{\le} \frac{M}{\sqrt{n}} \limninf 0,
  \end{equation}
  where~$M = \supl_{t \ge 0} |\phi(t)|$. Hence, we conclude that the limiting process~$G$ is a.s. bounded continuous (e.g., see P.~Billingsley~\cite[Theorem 13.4]{B1999book}) and thus~$X_n^{(2)}$ converges to~$G$ in distribution in~$D[0,+\infty]$ with the \textit{uniform} topology.

  Finally, Slutsky's theorem together with~\eqref{eq:proof_thm1_sn1_dconv} and~\eqref{eq:thm1_proof_s2_Dconv} implies~\eqref{eq:conv_skor_th_fclt1}. The formula~\eqref{eq:thm1_conv_of_first_moment} is supplied by~\eqref{eq:prop_S1_conv_eq2} and~\eqref{eq:fidi_prop_eq1}. The theorem is proven.
\end{proof}

\section{Functional central limit theorem for the first hitting time}
\label{sec:FCLT_hitting_time}
Recall the definition of the first-hitting time~$Q_n$ in~\eqref{eq:Qn_def}. The main result of this section is the following theorem.
\begin{theorem}[Functional CLT for the first-hitting time]
  \label{thm:thm2}
  In notation of Theorem~\ref{thm:thm1}, let~$\phi>0$ be a bounded function with a locally bounded first derivative. Set~$L = m_1(+\infty)$, where~$m_1$ is defined in~\eqref{eq:def_mk_omega}, and~$\tau = m_1^{-1}$. Then,
  \begin{equation}
    \label{eq:thm2_claim}
    \sqrt{n}\left(
      \begin{bmatrix}
        S_n\\
        Q_n
      \end{bmatrix}
      -
      \begin{bmatrix}
        m_1\\
        \tau
      \end{bmatrix}    
    \right)\limninfdist
    \begin{bmatrix}
      G\\
      -\tau' \cdot G \circ \tau
    \end{bmatrix}
  \end{equation}
  in the sense of convergence in distribution of random elements of~$D[0, +\infty) \times D[0,L)$ equipped with the product of the corresponding locally uniform topologies. 
\end{theorem}
\begin{remark}
  The theorem implies the convergence of finite-dimensional joint distributions. The law of large numbers for the second component also takes place,
  \begin{equation}
    Q_n \limninfprob \tau,
  \end{equation}
  uniformly on compacts in~$[0, L)$.
\end{remark}

The proof of Theorem~\ref{thm:thm2} relies heavily on Skorohod's almost sure representation theorem (e.g., see P.~Billingsley~\cite[Theorem~6.7]{B1999book}). We recall that, under mild conditions, Skorohod's theorem allows one to couple on the same probability space weakly convergent sequences of probability measures so that the convergence of the corresponding random elements becomes almost sure. First, we answer the question of what happens to~$Q_n(h)$ as~$n \to \infty$ if~$h \ge L$.
\begin{lemma}
  \label{lem:hitting_time_h_large}
  \begin{equation}
    Q_n(h) \limninfprob +\infty, \quad h \ge L.
  \end{equation}
\end{lemma}
\begin{proof}
  We recall Theorem~\ref{thm:thm1}. Since for
  \begin{equation}
    X_n(t) = \sqrt{n} \big(S_n(t)- \mexp{S_n(t)}\big),\quad n \in \mathbb{N},
  \end{equation}
  we have the weak convergence~\eqref{eq:conv_skor_th_fclt1}, the Skorohod representation theorem allows us to couple on the same probability space the random elements~$\{X_n\}$ along with~$G$ so that
  \begin{equation}
    \label{eq:lem_hittingtime_proof_eq1}
    X_n \limninfas G
  \end{equation}
  uniformly on compacts. Again, we use the same symbols to denote the coupled objects.
  
  Since~$Q_n(h)$ is non-decreasing, it suffices to prove
  \begin{equation}
    \label{eq:hitting_time_h_large_proof_eq1}
    Q_n(L) \limninfas +\infty.
  \end{equation}
  Suppose that~\eqref{eq:hitting_time_h_large_proof_eq1} does not hold. Then, there is finite~$T>0$ and a subsequence of indices~$\{n_k\}$ such that
  \begin{equation}
    {Q}_{n_k}(L) \le T, \quad k=1,2,\ldots.
  \end{equation}
  By the definition~\eqref{eq:Qn_def} we have
  \begin{equation}
    {S}_{n_k}(T) \ge L.
  \end{equation}
  Passing to the limit~$k \to \infty$ and noticing that
  \begin{equation}
    {S}_n(T) \limninfas m_1(T),
  \end{equation}
  we obtain
  \begin{equation}
    m_1(T) \ge L = m_1(+\infty),
  \end{equation}
  which is impossible since~$m_1$ is \textit{strictly} increasing. This concludes the proof.
\end{proof}

Heuristically, if the level~$h$ is higher or equal to~$L$, then asymptotically~$(S_n(t),t \ge 0)$ either reaches it in infinite time or never at all.

\begin{proof}[Proof of Theorem~\ref{thm:thm2}]
  Following the same line of reasoning as in the proof of Lemma~\ref{lem:hitting_time_h_large}, however relying on Corollary~\ref{cor:thm1} instead of Theorem~\ref{thm:thm1}, we will have
  \begin{equation}
    \label{eq:thm_FCLT_Qn_eq0}
    \sqrt{n} (S_n- m_1) \limninfas G
  \end{equation}
  uniformly on compacts. By enriching the probability space, if necessary, we can assume that the~$S_n$ have sample paths as prescribed by the right-hand side of~\eqref{eq:S_n_via_U}, enabling us to use the pathwise arguments below.

  It is convenient to work with the left-continuous analog of~$Q_n$ defined by
  \begin{equation}
    Q_n^-(h) \df Q_n(h-0) = \inf \{s \in [0, +\infty)|\, S_n(s) \ge h \}.
  \end{equation}
  Set
  \begin{equation}
    \widetilde{Q}_n =  m_1\circ Q_n^-, \quad \widetilde{G} = G \circ \tau, \quad \widetilde{S}_n = S_n \circ \tau
  \end{equation}
  on~$[0, L)$. We have that for every~$\epsilon \in (0,L)$
  \begin{equation}
    \label{eq:thm_FCLT_Qn_eq1}
    \|\sqrt{n}\big(\widetilde{S}_n-\id\big)- \widetilde{G}\|_\epsilon \limninfas 0,
  \end{equation}
  where~$\|\cdot \|_\epsilon$ is the uniform norm over~$[0,L-\epsilon]$.
  \\[2ex]
  \noindent \textbf{a)} We are going to prove that for every~$\epsilon \in (0,L)$
  \begin{equation}
    \label{eq:thm_FCLT_Qn_eq3}
    \|\sqrt{n} \big(\widetilde{Q}_n - \id\big) + \widetilde{G}\|_{\epsilon} \limninfas 0.
  \end{equation}
  
  We proceed with a line of argument similar to that in~G.~R.~Shorack~\cite[p. 151]{S1973}, extending it to our scenario. The triangle inequality implies
  \begin{equation}
    \label{eq:thm_FCLT_Qn_eq2}
    \begin{aligned}
      &\|\sqrt{n} \big(\widetilde{Q}_n - \mathrm{id}\big) + \widetilde{G} \|_\epsilon \le \sqrt{n} \|\widetilde{S}_n \circ \widetilde{Q}_n - \id \|_\epsilon +\|\widetilde{G} \circ \widetilde{Q}_n -  \widetilde{G}\|_\epsilon \\ 
      &+\|\sqrt{n}(\widetilde{S}_n \circ \widetilde{Q}_n - \widetilde{Q}_n) -\widetilde{G} \circ \widetilde{Q}_n\|_\epsilon\\
      &\le \sqrt{n} \|\widetilde{S}_n \circ \widetilde{Q}_n - \id \|_\epsilon + \omega_{\widetilde{G}}(\|\widetilde{Q}_n -  \id\|_\epsilon) + \|\sqrt{n}(\widetilde{S}_n  - \id) -\widetilde{G}\|_{\widetilde{\varepsilon}},
    \end{aligned}
  \end{equation}
  where~$\widetilde{\varepsilon} = \widetilde{Q}_n(L-\epsilon)$ and $\omega_{\widetilde{G}}(\cdot)$ is the modulus of continuity of~$\widetilde{G}$.

  Since~$\widetilde{S}_n(L-\frac{\epsilon}{2}) \limninfas L - \frac{\epsilon}{2}$, for every realization and for all~$n$ large enough (depending on that realization), we will have~$\widetilde{S}_n(L-\frac{\epsilon}{2})>L-\epsilon$. Thus~$\widetilde{\varepsilon} \le L- \frac{\varepsilon}{2}$, and the third term of the last inequality converges to zero due to~\eqref{eq:thm_FCLT_Qn_eq1}. Also, $\|\widetilde{S}_n \circ \widetilde{Q}_n - \id \|_\epsilon$ is bounded by the maximal jump of~$S_n$, which is less or equal to~$M/n$. We get the bound
  \begin{equation}
    \sqrt{n} \big\|\widetilde{S}_n \circ \widetilde{Q}_n - \id \big\|_\epsilon \le \frac{M}{\sqrt{n}},
  \end{equation}
  where~$M = \supl_{x \ge 0}\phi(x)$. This allows us to make the first term on the right-hand side of the last inequality in~\eqref{eq:thm_FCLT_Qn_eq2} arbitrarily small by choosing~$n$ large.
  
  Since~$\widetilde{G}$ is uniformly continuous on~$[0,L-\epsilon]$ (in fact, on~$[0,L]$!), the last thing we need to prove is that~$\|\widetilde{Q}_n -  \id\|_\epsilon \limninf 0$ almost surely. By symmetry and since~$\widetilde{Q}_n$ is left-continuous,
  \begin{equation}
    \label{eq:proof_th2_eq1}
    \supl_{h \in [0, \widetilde{S}_n(L-\epsilon/2)]}|\widetilde{Q}_n(h) -  h| = \supl_{s \in [0, \widetilde{Q}_n \circ \widetilde{S}_n(L-\epsilon/2)]}|\widetilde{S}_n(s) -  s|.
  \end{equation}
  The graphs of~$\widetilde{S}_n$ and~$\widetilde{Q}_n$ over the specified intervals are the reflections of each other about the diagonal, which explains~\eqref{eq:proof_th2_eq1}. Since~$\widetilde{S}_n(L-\epsilon/2) > L - \epsilon$, the identity~\eqref{eq:proof_th2_eq1} implies
  \begin{equation}
    \|\widetilde{Q}_n -  \id\|_\epsilon \le \supl_{s \in [0, \widetilde{Q}_n \circ \widetilde{S}_n(L-\epsilon/2)]}|\widetilde{S}_n(s) -  s| \le \|\widetilde{S}_n -  \id\|_{\epsilon/2}.
  \end{equation}
  From~\eqref{eq:thm_FCLT_Qn_eq1} we see that the latter term converges to zero almost surely, so the convergence~\eqref{eq:thm_FCLT_Qn_eq3} is established.
  \\[2ex]
  \noindent \textbf{b)} In order to return to the original process~$Q_n$ (without tilde), we need to establish 
  \begin{equation}
    \label{eq:thm_FCLT_Qn_eq4}
    \| \sqrt{n}\big(Q_n^- - \tau\big) + \tau' \cdot \widetilde{G}\|_\epsilon \limninfas 0.
  \end{equation}
  To that end, write
  \begin{equation}
    \| \sqrt{n}\big(Q_n^- - \tau\big) + \tau' \cdot \widetilde{G}\|_\epsilon =  \| H_n \cdot \sqrt{n}(\widetilde{Q}_n - \id) + \tau' \cdot \widetilde{G}\|_\epsilon,
  \end{equation}
  where~$H_n$ is the function defined by
  \begin{equation}
    \label{eq:clt_hitting_time_proof_Hn_def}
    H_n(h) = 
    \begin{cases}
      \frac{\tau\circ \widetilde{Q}_n(h)-\tau(h)}{\widetilde{Q}_n(h) - h},  &\widetilde{Q}_n(h) - h \ne 0,\\
      \tau'(h), &\widetilde{Q}_n(h) - h = 0.
    \end{cases}
  \end{equation}
  Rewrite
  \begin{equation}
    \sqrt{n}\big(Q_n^- - \tau\big) + \tau' \cdot \widetilde{G} =(H_n - \tau' )\cdot \sqrt{n}(\widetilde{Q}_n - \id) + \tau'  \cdot (\sqrt{n}(\widetilde{Q}_n - \id)+\widetilde{G}).
  \end{equation}
  Since~$\tau'$ is bounded and~$\sqrt{n}(\widetilde{Q}_n - \id)$ is a.s. convergent by a) and thus also bounded, the triangle inequality yields
  \begin{equation}
    \|\sqrt{n}\big(Q_n^- - \tau\big) + \tau' \cdot \widetilde{G}\|_\epsilon \le C \left(\|H_n -\tau'\|_\epsilon +\|\sqrt{n}(\widetilde{Q}_n - \id)+\widetilde{G} \|_\epsilon\right),
  \end{equation}
  for some constants~$C>0$. The second term converges to zero almost surely by a).

  The function~$\tau$ is twice continuously differentiable on~$[0,L-\epsilon]$ (but not on~$[0,L]$!). The Lagrange mean-value theorem of the second order applied to~$\tau$ in~\eqref{eq:clt_hitting_time_proof_Hn_def} implies
  \begin{equation}
    \|H_n -\tau'\|_\epsilon  \le \widetilde{C} \|\widetilde{Q}_n-\id\|_\epsilon
  \end{equation}
  for some constant~$\widetilde{C}>0$ independent of~$n$.
  The last quantity converges to zero almost surely as was established in a).

  Finally, by the left-continuity, in~\eqref{eq:thm_FCLT_Qn_eq4} we can take the supremum over~$[0,L-\varepsilon)$, which is the same for the original process~$Q_n$. Thus,
  \begin{equation}
    \supl_{h \in [0, L-\varepsilon)} |\sqrt{n}\big(Q_n - \tau\big) + \tau' \cdot \widetilde{G}| \limninfas 0,
  \end{equation}
  and the locally uniform convergence on~$[0,L)$ follows. Since this holds on the same probability space where~\eqref{eq:thm_FCLT_Qn_eq0} does, the \textit{vector} convergence in~\eqref{eq:thm2_claim} is established.
\end{proof}

\begin{acks}
	The author thanks Gernot Akemann, who suggested Remark~\ref{rem:Gernot}, as well as the anonymous referees, whose feedback and additional references helped to improve the presentation of the paper. The work was supported by the Research Foundation -- Flanders (FWO), project 12K1823N.
\end{acks}

\end{document}